\documentclass{amsart}

\usepackage{amsmath}
\usepackage{amssymb}
\usepackage{amscd, mathtools, comment}
\usepackage[colorlinks,pagebackref=true]{hyperref}
\usepackage[all]{xy}

\usepackage{enumerate}

\numberwithin{equation}{section}

\newcommand{\old}[1]{{\color{red} #1}}

\def\pd{\operatorname{pd}}
\def\Ext{\operatorname{Ext}}
\def\Tor{\operatorname{Tor}}
\def \rfed{\operatorname{rfed}}
\def \extdeg {\operatorname{ext.deg}} 
\def\Ker{\operatorname{Ker}}

\newcommand{\fkm}{\mathfrak{m}}
\newcommand{\fkn}{\mathfrak{n}}

\newcommand{\fkp}{\mathfrak{p}}
\newcommand{\fkq}{\mathfrak{q}}

\newcommand{\fkM}{\mathfrak{M}}
\newcommand{\fkN}{\mathfrak{N}}

\newcommand{\rme}{\mathrm{e}}

\newcommand{\rmr}{\mathrm{r}}

\newcommand{\rmv}{\mathrm{v}}

\def\depth{\operatorname{depth}}

\def\Ext{\operatorname{Ext}}

\def\height{\operatorname{ht}}

\def\Hom{\operatorname{Hom}}

\def\m{\mathfrak{m}}

\def\p{\mathfrak{p}}
 
\def\rhom{\operatorname{\mathbf{R}Hom}}

\def\spec{\operatorname{Spec}}

\def\syz{\mathrm{\Omega}}

\def\Tor{\operatorname{Tor}}

\def \fd {\operatorname{fd}}
\def \pd {\operatorname{pd}}    
\def \rfd {\operatorname{Rfd}}
\def \D {\mathcal D}
\def \n {\mathfrak n}     
\def \lotimes {\otimes^{\mathbf L}} 

\def\@citecolor{blue}
\def\@linkcolor{blue}
\def\@urlcolor{blue}
\def\@urlcolor{blue}

\def\depth{\operatorname{depth}}
\def\pd{\operatorname{pd}}
\def\Hom{\operatorname{Hom}}

\def\extdeg{\operatorname{ext.deg}}
\def\sup{\operatorname{sup}}
\def\uac1{\operatorname{uac1}}
\theoremstyle{plain} 
\newtheorem{Theorem}{Theorem}[section]
\newtheorem{Lemma}[Theorem]{Lemma}
\newtheorem{Corollary}[Theorem]{Corollary}

\newtheorem{thm}[Theorem]{Theorem} 
\newtheorem{lem}[Theorem]{Lemma}
\newtheorem{prop}[Theorem]{Proposition}
\newtheorem{cor}[Theorem]{Corollary}

\theoremstyle{definition} 
\newtheorem{Remark}[Theorem]{Remark}

\newtheorem{Example}[Theorem]{Example}

\newtheorem{Definition}[Theorem]{Definition}

\newtheorem{fact}[Theorem]{Fact} 

\newtheorem{dfn}[Theorem]{Definition}

\newtheorem{quest}[Theorem]{Question}
\newtheorem{rem}[Theorem]{Remark}

\newtheorem{ex}[Theorem]{Example}
\newtheorem{condition}[Theorem]{Condition}

\newtheorem*{convention}{Convention}
\newtheorem*{ac}{Acknowledgments}

\newcommand{\om}{\Omega}
\let\epsilon\varepsilon
\let\phi=\varphi
\let\kappa=\varkappa

\title{On a generalized Auslander-Reiten conjecture}
\author{Souvik Dey}  
\address{Souvik Dey: Department of Mathematics \\ University of Kansas\\405 Snow Hall, 1460 Jayhawk Blvd.\\ Lawrence, KS 66045, U.S.A.}
\email{souvik@ku.edu}  

\author{Shinya Kumashiro}
\address{Shinya Kumashiro: National Institute of Technology, Oyama College, 
	771 Nakakuki, Oyama, Tochigi, 323-0806, Japan}
\email{skumashiro@oyama-ct.ac.jp}

\author{Parangama Sarkar}
\address{Parangama Sarkar:  Department of Mathematics,
		Indian Institute of Technology, Palakkad, India}
\email{parangama@iitpkd.ac.in}

\thanks{The second author was supported by JSPS KAKENHI Grant Number JP21K13766.}
\thanks{The third author was supported by SERB POWER Grant with Grant No. SPG/2021/002423.}

\begin{document}

\begin{abstract}
It is well-known that the generalized Auslander-Reiten condition (GARC) and the symmetric Auslander condition (SAC) are equivalent, and (GARC) implies that the Auslander-Reiten condition (ARC). In this paper we explore (SAC) along with the several canonical change of rings $R \to S$. First, we prove the equivalence of (SAC) for $R$ and $R/xR$, where $x$ is a non-zerodivisor on $R$, and the equivalence of (SAC) and (SACC) for rings with positive depth, where (SACC) is the symmetric Auslander condition for modules with constant rank. The latter assertion affirmatively answers a question posed by Celikbas and Takahashi. 

Secondly, for a ring homomorphism $R \to S$, we prove that if $S$ satisfies (SAC) (resp. (ARC)), then $R$ also satisfies (SAC) (resp. (ARC)) if the flat dimension of $S$ over $R$ is finite. We also prove that (SAC) for $R$ implies that (SAC) for $S$ when $R$ is Gorenstein and $S=R/Q^\ell$, where $Q$ is generated by a regular sequence of $R$ and the length of the sequence is at least $\ell$. 
This is a consequence of more general results about Ulrich ideals proved in this paper.
Applying these results to determinantal rings and numerical semigroup rings, we provide new classes of rings satisfying (SAC). A relation between (SAC) and an invariant related to the finitistic extension degree is also explored.
\end{abstract}  

\maketitle
\section{Introduction}  

The purpose of this paper is to study the following condition for commutative Noetherian rings $R$.

	\begin{condition}\label{b1.1} 
		The equality 
		\[
		\pd_R M= \sup \{ i\in \mathbb{Z} \mid \Ext_R^i (M, M\oplus R)\ne0\}
		\] 
		hold for all finitely generated $R$-modules $M$.
	\end{condition}

It is known that there are several equivalent conditions of the above condition (see Lemma \ref{b2.1} or \cite[Theorem 1.2]{CT}). These equivalent conditions are known as the {\it symmetric Auslander condition} (SAC) and the {\it generalized Auslander-Reiten condition} (GARC) and are studied in the literature (\cite{CT, CHmanu, D, HSV, W1, W2}). Here, ``generalized'' of (GARC) means that (GARC) can be observed as a generalization of the condition appearing in the famous conjecture of Auslander-Reiten (\cite{AR}), which is called the {\it Auslander-Reiten condition} (ARC).

\vspace*{1mm}
(ARC): For each finitely generated $R$-module $M$, if $\Ext_R^{>0} (M, M\oplus R)=0$, then $M$ is projective.

\vspace*{1mm}
The {\it Auslander-Reiten conjecture} (\cite{AR}) claims that (ARC) holds for all Noetherian rings $R$. Although various special cases of this conjecture are known to hold true, it remains open in its full generality. For the progress on the Auslander-Reiten conjecture and related conjectures, see, for example, \cite[Introduction]{GT}. Here, we only note that (SAC) (equivalently, (GARC) and Condition \ref{b1.1}) implies that (ARC) (see \cite[Theorem 4.8]{CT} or this is easily observed by the definitions of Condition \ref{b1.1} and (ARC)). Thus, it is worth to find a class of rings satisfying (SAC) in study of (ARC).
	

Let us explain our results in this paper. We explore (SAC) from the aspect of change of rings. We then see that (SAC) is preserved by several canonical change of rings in spite of the fact that they are not known for (ARC). Roughly speaking, (SAC) may be flexible than (ARC). In conclusion, for example, we answer affirmatively a question on (SAC) posed by Celikbas and Takahashi (\cite[Question 4.9]{CT}) and strengthen the results \cite[Proposition 2.11 and Theorem 2.16]{K} on (ARC) for determinantal rings by replacing (ARC) with (SAC). 
To explain our results in detail, hereafter in this section, let $(R, \fkm)$ be a Noetherian local ring with maximal ideal $\fkm$. 

In Section \ref{section2}, we consider the change of rings $R \to R/xR$, where $x\in \fkm$ is a non-zerodivisor on $R$. In particular, we obtain the following. Recall that the {\it symmetric Auslander condition for modules with constant rank} (SACC) is the condition restricting ``finitely generated modules'' in (SAC) to ``finitely generated modules with constant rank'' (see the end of this section or \cite[Introduction]{CT}). 

\begin{Theorem} {\rm (Theorem \ref{SACCtoSAC})}
	Suppose that $R$ has a positive depth and $x\in \fkm$ is a non-zerodivisor on $R$.  If $R$ satisfies {\rm (SACC)}, then $R/xR$ satisfies {\rm (SAC)}. 
\end{Theorem}	

As an immediate consequence of the above, we see that (SAC) for $R$ and $R/xR$ are equivalent (Corollary \ref{a0.3}). We should emphasize that this is not known for (ARC) without assuming that $R$ is complete (see, for example, \cite[Proposition 2.2]{K}).
In addition, we give a complete answer of a question posed by Celikbas-Takahashi \cite[Question 4.9]{CT} that asserts the equivalence of (SAC) and (SACC) for rings $R$ with positive depth (Corollary \ref{b2.4}).
We also investigate (SACC) for more details by showing that, without loss of equivalence, we can further restrict ``finitely generated modules with constant rank'' to ``finitely generated modules with a fixed constant rank $n\ge 0$'' (Corollary \ref{c29}). 

The aim of Section \ref{section3} is to prove the following.

\begin{thm}{\rm (Theorem \ref{mainsec3})} \label{1.3}
Let $R \to S$ be a homomorphism of local rings such that $\fd_R S<\infty$. Then the following hold. 
\begin{enumerate}[\rm(1)]

\item If $S$ satisfies {\rm (SAC)}, then $R$ satisfies {\rm(SAC)}.

\item If $S$ satisfies {\rm (ARC)}, then $R$ satisfies {\rm (ARC)}.   
\end{enumerate}
\end{thm}

We note that Theorem \ref{1.3} generalizes \cite[Theorem 4.5 (1)]{CT} and \cite[Theorem 2.6 (3) $\Rightarrow$ (1)]{K} since $\fd_R R/Q^\ell$ is finite, where $Q$ is an ideal of $R$ generated by a regular sequence of $R$ and $\ell>0$.
It might be worth to note that, using the hereditary properties of (SAC) observed in Sections \ref{section2} and \ref{section3}, the assertion that (SAC) holds for all Noetherian (resp. Cohen-Macaulay or Gorenstein) local rings containing a field can be reduced to the case of either factorial domains or rings having isolated singularity (Remark \ref{rem211}). 

In Section \ref{section4}, we consider the change of rings $R \to R/I^\ell$, where $I$ is an Ulrich ideal of $R$ and $\ell>0$. Here, the notion of Ulrich ideals is introduced by Goto, Ozeki, Takahashi, Watanabe, and Yoshida (\cite{GOTWY}), and it is well-explored, for example, see \cite{GTT2, K} and the references therein. The definition of Ulrich ideals is as follows.

\begin{Definition}(\cite[Definition 2.1]{GOTWY})\label{Ulrich}
Let $(R, \fkm)$ be a Cohen-Macaulay local ring and $I$ an $\fkm$-primary ideal of $R$. We say that $I$ is an {\it Ulrich ideal} of $R$ if $I^2=Q I$ for some parameter ideal $Q \subseteq I$ and $I/I^2$ is a free $R/I$-module. 
\end{Definition}

We note that $\fd_R R/I=\infty$ if an Ulrich ideal $I$ is not a parameter ideal (\cite[Theorem 1.4]{GOTWY}); hence, this is not the case of Theorem \ref{1.3}.
Using the notion of Ulrich ideals, we observe hereditary properties among $R$, $R/I$, and $R/I^\ell$  (Theorem \ref{ul}). By applying this result, we obtain an equivalence of (SAC) between $R$ and $R/Q^\ell$, where $Q$ is an ideal generated by a regular sequence on $R$ of length $\ge \ell$ (Theorem \ref{a0.5}). Moreover, we strengthen the results \cite[Proposition 2.11 and Theorem 2.16]{K} on (ARC) for determinantal rings by replacing (ARC) with (SAC) (Corollaries \ref{c411} and \ref{det52}). 

In Section \ref{sectionglue}, we provides a new class of rings satisfying (SAC) arising from the gluing of numerical semigroup rings. 

In Section \ref{section6}, we concerns a new invariant which raised from the {\it finistic extension degree} defined by Diveris (\cite[Definition 2.2]{D}). 
Diveris proves that for a Gorenstein ring $R$, the finistic extension degree is finite if and only if $R$ satisfies (SAC) (\cite[Corollary 3.2]{D}). We here show that our modified invariant also behaves well as finistic extension degree in Noetherian rings (Theorem \ref{fedth}).

\begin{convention}
	Throughout this paper, all rings are commutative Noetherian rings with identity. We denote by $(R, \fkm)$ a local ring with  maximal ideal $\fkm$. In Sections \ref{section2},  \ref{section4}, and \ref{sectionglue}, all modules are finitely generated.  
	For each finitely generated module $M$ over a local ring $R$, $\mu_R (M)$ (resp. $\ell_R (M)$, $\om^n_R M$) denotes the number of elements in a minimal system of generators of $M$ (resp. the length of $M$, the $n$th syzygy of $M$ in a minimal $R$-free resolution of $M$). 
	If $M$ is a Cohen-Macaulay $R$-module, $\rmr_R (M)$ denotes the Cohen-Macaulay type of $M$. Let $\rmv(R)$ denote the embedding dimension of $R$. 
	
	For convenience, $\Ext_R^{>n} (M, N)=0$ (resp. $\Ext_R^{\gg0} (M, N)=0$) stands for the vanishing of $\Ext_R^{i} (M, N)$ for all $i>n$ (resp. the vanishing of $\Ext_R^{i} (M, N)$ for all large enough $i>0$).

For the convenience of readers, we summarize here the conditions that we consider throughout this paper.
\begin{enumerate}
\item[{\bf (ARC):}] For each finitely generated $R$-module $M$, if $\Ext_R^{>0} (M, M\oplus R)=0$, then $M$ is projective.

\item[{\bf (GARC):}] For each finitely generated $R$-module $M$, if $\Ext_R^{\gg0} (M, M\oplus R)=0$, then $\pd_R M<\infty$.

\item[{\bf (SAC):}] For each finitely generated $R$-module $M$, if \[
\Ext_R^{>0} (M, R)=\Ext_R^{\gg0} (M, M)=0,
\] 
then $\Ext_R^{>0} (M, M)=0$.

\item[{\bf (SACC):}] For each finitely generated $R$-module $M$ with constant rank, if 
\[
\Ext_R^{>0} (M, R)=\Ext_R^{\gg0} (M, M)=0,
\] 
then $\Ext_R^{>0} (M, M)=0$.
\end{enumerate}
\end{convention}

\begin{rem} \label{rem1}   
	For a ring $R$, the conditions (SAC) and (GARC) are equivalent even if $R$ is not local. Indeed, the only place the local assumption is used in the proof of \cite[Theorem 4.8]{CT} is to invoke \cite[Lemma 1(iii), page 154]{Mat} which says that for a finitely generated $R$-module $M$, if $\pd_R M=n<\infty$ then $\Ext^n_R(M,R)\ne 0$. This however goes through in the non-local case as well since if $\pd_R M=n<\infty$ then $\pd_{R_{\m}} M_{\m}=n$ for some maximal ideal $\m$, by \cite[8.52]{W}, and hence $\Ext^n_R(M,R)_{\m}\cong \Ext^n_{R_{\m}}(M_{\m},R_{\m})\ne 0$, so $\Ext^n_R(M,R)\ne 0$. We use this observation without referencing.

\end{rem}

\section{(SAC) and other related conditions} \label{section2}
 



 
Let $R$ be a local ring. 

\begin{Lemma}\label{b2.1}
The following conditions are equivalent:
\begin{enumerate}[{\rm (1)}] 
\item $\pd_R M= \sup \{ i\in \mathbb{Z} \mid \Ext_R^i (M, M\oplus R)\ne0\}$ holds for each finitely generated $R$-module M.
\item {\bf (GARC):} For each finitely generated $R$-module $M$, if $\Ext_R^{\gg0} (M, M\oplus R)=0$ then $\pd_R M<\infty$.
\item {\bf (SAC):} For each finitely generated $R$-module $M$, if $\Ext_R^{>0} (M, R)=\Ext_R^{\gg0} (M, M)=0$ then $\Ext_R^{>0} (M, M)=0$.
\end{enumerate}
\end{Lemma}

\begin{proof}
(1) $\Rightarrow$ (2): This is clear.

(2) $\Rightarrow$ (1): Let $M$ be a finitely generated $R$-module. If the projective dimension of $M$ is infinite, then so is $\sup \{ i\in \mathbb{Z} \mid \Ext_R^i (M, M\oplus R)\ne0\}$ by hypothesis. Hence the equality in the assertion (1) holds. Assume that the projective dimension of $M$ is finite and set $n=\pd_R M$. Thus $\Ext_R^i (M, M\oplus R)=0$ for all $i>n$. Since by \cite[Lemma 1 (iii), page 154]{Mat} and the discussion in the end of the previous section, $\Ext_R^n (M, R)\ne0$, we get $\sup \{ i\in \mathbb{Z} \mid \Ext_R^i (M, M\oplus R)\ne0\}=n$.

(2) $\Leftrightarrow$ (3): This follows from \cite[Theorem 4.8]{CT}.
\end{proof}


\begin{Remark}\label{zzz2.2}
Along the same lines of the proof of Lemma \ref{b2.1}, one can also obtain the equivalent conditions of (SACC) by substituting ``finitely generated $R$-module" for ``finitely generated $R$-module that has  constant rank" in Lemma \ref{b2.1}.
\end{Remark}

\begin{prop}{\rm (\cite[Proposition 2.3]{cd}, \cite[11.65]{W})}\label{long}
Let $(R, \fkm)$ be a local ring and $x\in \fkm$ a non-zerodivisor on $R$. Set $\overline R=R/(x)$ and let $M,$ $N$ be $\overline R$-modules. Then there is the following long exact sequence.
\begin{align*}
0\to \Ext^{1}_{\overline R}(M,N)\to &\Ext^{1}_R(M,N)\to \Ext^{0}_{\overline R}(M,N)\\
\vdots& \\
\to \Ext^{n}_{\overline R}(M,N)\to &\Ext^{n}_R(M,N)\to \Ext^{n-1}_{\overline R}(M,N)\\
\to \Ext^{n+1}_{\overline R}(M,N)\to &\Ext^{n+1}_R(M,N)\to \Ext^{n}_{\overline R}(M,N)\to\cdots.
\end{align*}
	
\end{prop}

\begin{Theorem} \label{SACCtoSAC}
	Let $(R,\mathfrak m)$ be a local ring with positive depth and $x\in \fkm$ be a non-zerodivisor on $R$.  If $R$ satisfies {\rm (SACC)}, then $R/xR$ satisfies {\rm (SAC)}. 
\end{Theorem}		
	\begin{proof}
Let $\overline R=R/(x)$ and we use the notation $\overline{(-)}$ for $(-)\otimes_R R/(x)$.
Let $M$ be a finitely generated $\overline R$-module such that $\Ext^{>0}_{\overline R}(M,\overline R)=0$ and $\Ext^{\gg 0}_{\overline R}(M,M)=0$. We show that $\Ext^{>0}_{\overline R}(M,M)=0$.     

By looking at the long exact sequence obtained from Proposition \ref{long} with $N=M$, $\Ext^{\gg 0}_{\overline R}(M,M)=0$ implies that $\Ext^{\gg 0}_{R}(M,M)=0$. Hence, $\Ext^{\gg 0}_R(\syz_R^1 M,M)=0$. 
Similarly, applying Proposition \ref{long} with $N=\overline{R}$, $\Ext^{>0}_{\overline R}(M,\overline R)=0$ shows that $\Ext^{>1}_R(M,R)=0$.
Applying $\Hom_R(\syz_R^1 M, -)$ to a part of minimal free resolution 
\[
0\to \syz_R^1 M\to R^{\oplus \mu_R(M)}\to M\to 0
\] 
of $M$, we get $\Ext^{\gg 0}_R(\syz_R^1 M,\syz_R^1 M)=0$ since $\Ext^{>0}_R(\syz_R^1 M,R)=0$ and $\Ext^{\gg 0}_R(\syz_R^1 M,M)=0$. Since $R$ satisfies (SACC) and $\syz_R^1 M$ has constant positive rank over $R$ (as $M$ has constant rank $0$ over $R$), we get $\Ext^{>0}_R(\syz_R^1 M, \syz_R^1 M)=0$. 

On the other hand, since $R$ is of positive depth, by \cite[Corollary 4.6]{CT},  $R$ satisfies (ARC). Hence, $\Ext^{>0}_R(\syz_R^1 M,R)=\Ext^{>0}_R(\syz_R^1 M, \syz_R^1 M)=0$ implies that $\syz_R^1 M$ is a free $R$-module. It follows that $\pd_R M\le 1$.  Therefore $\Ext^{>1}_R(M,M)=0$. Again using Proposition \ref{long}, we get $\Ext^{n-1}_{\overline R}(M,M)\cong \Ext^{n+1}_{\overline R}(M,M)$ for all $n\geq 2$. Since $\Ext^{\gg 0}_{\overline R}(M,M)=0$, we get $\Ext^{>0}_{\overline R}(M,M)=0$. It follows that $R/xR$ satisfies {\rm (SAC)}. 
\end{proof}

As a consequence of the above theorem we prove the following.

\begin{Corollary}\label{a0.3}
Let $(R,\mathfrak m)$ be a  local ring with positive depth and $x\in \fkm$ be a non-zerodivisor on $R$. Then the following are equivalent:
\begin{enumerate}[{\rm (1)}] 
\item $R$ satisfies {\rm (SAC)}.
\item $R/xR$ satisfies {\rm (SAC)}.
\end{enumerate}
\end{Corollary}

\begin{proof}
	$(1) \Rightarrow (2)$ follows from Theorem \ref{SACCtoSAC}.
	$(2) \Rightarrow (1)$ follows from \cite[Theorem 4.5(1)]{CT} (see also Theorem \ref{mainsec3}).     
\end{proof}   

The following answers a question of Celikbas and Takahashi (\cite[Question 4.9]{CT}).  

\begin{Corollary}\label{b2.4}
	Let $(R,\mathfrak m)$ be a  local ring with positive depth. Then 
	$R$ satisfies {\rm (SAC)} if and only if $R$ satisfies {\rm (SACC)}.
\end{Corollary}

\begin{proof}
	(SAC) implies (SACC) by definition. Conversely, if $R$ satisfies (SACC), then $R/xR$ satisfies (SAC) by Theorem \ref{SACCtoSAC}, where $x$ is a non-zerodivisor on $R$. Hence $R$ satisfies (SAC) by Corollary \ref{a0.3}.
\end{proof}

\begin{cor}\label{poseries} Let $R$ be a local ring and $T$ be an indeterminate. Then $R$ satisfies $\rm{(SAC)}$ if and only if $R[[T]]$ satisfies $\rm{(SAC)}$.  
\end{cor}   

\begin{proof} Since $R[[T]]$ is a local ring with regular element $T$ such that $R[[T]]/(T)\cong R$, the claim follows from Corollary \ref{a0.3}.  
\end{proof}    

Based on Corollary \ref{b2.4}, we further consider the following condition. Let $n\ge 0$ be an integer.
\vspace*{1mm}\\
(${\rm SACC}_n$) For every finitely generated $R$-module $M$ of constant rank $n$, $\Ext^{>0}_R(M,R)=0=\Ext^{\gg 0}_R(M,M)$ implies $\Ext^{>0}_R(M,M)=0$
\vspace*{1mm}

In the following, for a finitely generated module $M$ over a not necessarily local ring $R$, by $\syz^n_R M$ we will denote an $n$-th syzygy in some finitely generated free resolution of $M$.

\begin{prop}\label{saccrank} $R$ satisfies {\rm (SACC)} if and only if there exists an integer $n\ge 0$ such that $R$ satisfies ${\rm (SACC_n)}$. 
%
%
%
%
\end{prop} 

\begin{proof} It is enough to show that if $R$ satisfies ${\rm (SACC_n)}$ for some $n\geq 0$, then $R$ satisfies ${\rm (SACC)}$. Let $M$ be a nonzero finitely generated module of  constant rank such that $\Ext^{>0}_R(M,R)=0=\Ext^{\gg 0}_R(M,M)$. Suppose that $\pd_R M=\infty$. Thus, $\syz_R^jM\ne 0$ for all $j\geq 0$. For all $j\geq 1$, set $e_j=\text{rank } \om^j_R M$. Since each $\syz_R^j M$ is torsion-free and non-zero, $e_j\ge 1$. We also have $$\Ext_R^{>0}(\om^j_R M, R)=0\mbox{, } \Ext_R^{\gg 0}(\om^j_R M, \om^j_R M)=0.$$ 
	Now let $X:=(\om^1_R M)^{\oplus n}\oplus \om^{2}_R M$. Then $X$ has rank $e=ne_1+e_{2}\geq n+1$ and $X$ has a free submodule of rank $e$ (\cite[Proposition 1.4.3]{BH}). Thus we can consider a short exact sequence of $R$-modules 
	\begin{align}\label{eq51}
	0\longrightarrow R^{\oplus{(e-n)}}\longrightarrow X\longrightarrow L\longrightarrow 0
	\end{align}
	where $L$ has rank $n$.  
  Since $\Ext^{>0}_R(X,R)=0=\Ext^{\gg 0}_R(X,X)$, applying $\Hom_R(X,-)$ to \eqref{eq51}, we get $\Ext^{\gg 0}_R(X,L)=0.$ Again applying $\Hom_R(-,L)$ to \eqref{eq51}, we get $\Ext^{\gg 0}_R(L,L)=0.$ Now $\Ext^{> 0}_R(X,R)=0$ implies $\Ext^{> 0}_R(L,R)=0$. Since $R$ satisfies ${\rm (SACC_n)}$, it follows that $\Ext^{\geq 1}_R(L,L)=0.$ 
Applying  $\Hom_R(-,L)$ to \eqref{eq51}, we get $\Ext^{\geq 1}_R(X,L)=0.$ Finally applying $\Hom_R(X,-)$ to \eqref{eq51}, we get $\Ext^{\geq 1}_R(X,X)=0$  and hence 
$0= \Ext^{1}_R(\om^{1}_R M,\om^{2}_R M)$. Therefore
 $\om^{1}_R M$ is projective $R$-module and hence $\pd_R M<\infty$. Now $\Ext^{>0}_R(M,R)=0$ implies $M$ is a projective $R$-module and thus $R$ satisfies (SACC).
\end{proof}  

From Corollary \ref{b2.4} and Proposition \ref{saccrank} we get the following immediate consequence.  

\begin{cor}\label{c29}
 Let $R$ be a local ring of positive depth. Then $R$ satisfies {\rm (SAC)} if and only if $R$ satisfies ${\rm (SACC_n)}$ for some integer $n\ge 0$.
\end{cor}

\section{descent of (SAC) and (ARC) along local homomorphism of finite flat dimension} \label{section3}


Let $R$ be a ring. An $R$-module $M$ is said to have finite flat dimension if there exists an integer $n\ge 0$ and an exact sequence $0\to T_{n} \to \cdots \to T_0\to M\to 0$ where each $T_i$ is a flat $R$-module.  The aim of this section is to prove the following result.  

\begin{thm}\label{mainsec3} Let $R \to S$ be a homomorphism of local rings such that $\fd_R S<\infty$. Then the following hold. 
\begin{enumerate}[\rm(1)]

\item If $S$ satisfies {\rm (SAC)}, then $R$ satisfies {\rm(SAC)}.

\item If $S$ satisfies {\rm (ARC)}, then $R$ satisfies {\rm (ARC)}.   
\end{enumerate}
\end{thm}
As a consequence of the above theorem, we get the following.
\begin{cor}\label{mainseccor3} Let $R \to S$ be a homomorphism of local rings such that $\fd_R S<\infty$. Let $Q \to S$ be a surjective homomorphism of local rings whose kernel is generated by a $Q$-regular sequence. Then the following hold.  
\begin{enumerate}[\rm(1)]

\item If $Q$ satisfies {\rm (SAC)}, then $R$ satisfies {\rm(SAC)}.

\item Suppose that $Q$ is complete and $Q$ satisfies {\rm (ARC)}. Then $R$ satisfies {\rm (ARC)}.  
\end{enumerate}
\end{cor}
\begin{proof}(1) This follows from Corollary \ref{a0.3} and Theorem \ref{mainsec3} (1). 

(2) This follows from \cite[Proposition 2.2]{K} and Theorem \ref{mainsec3} (2). 
\end{proof}

We mention here that the statement of Corollary \ref{mainseccor3}(1) appears in \cite[Theorem 6.4]{AINSW}, but the proof is omitted with a remark that the proof is similar to the proof of \cite[Theorem 5.3]{AINSW}. Among other things, the proof of \cite[Theorem 5.3]{AINSW} relies crucially on the ring of cohomology operators and the structure of the graded Ext-algebra on it. In contrast, our proof relies only on basic homological algebra of complexes and gives a unified way for proving the descent of (SAC) and (ARC) along local ring map of finite flat dimension. Thus, one can see the proofs are different.

To prove Theorem \ref{mainsec3} we need some preparatory technical results.  

 We recall from \cite[(5.3.1) Definition, (5.3.6) Theorem]{ch} (and \cite[(A.3.12) Remark]{ch}) that for any $R$-module $M$, one has 
\begin{align*}
&\sup \{m\in \mathbb Z \mid \text{there exists a module } T \text{ of finite flat dimension such that } \Tor^R_m(M,T)\ne 0\}\\
=&\sup \{\depth R_{\p}-\depth_{R_{\p}} M_{\p} \mid \p \in \spec(R) \}. 
\end{align*} 

Following \cite[(1.0.1)]{ai}, we call the above common quantity the {\it restricted flat dimension} and denote by $\rfd_R M$. If $M\ne 0$, then $\rfd_R M\ge 0$. When $\dim R<\infty$, it is clear that $\rfd_R M\le \dim R<\infty$. In general, if $M$ is a finitely generated $R$-module, then $\rfd_R M<\infty$ by \cite[Theorem 1.1]{ai}. By convention, we put $\rfd_R 0=-\infty$. Note that $\rfd_R M\le 0$ if and only if $\Tor^R_{>0}(M,T)=0$ for every module $T$ of finite flat dimension. In the following, we record some basic observations about $\rfd_R(-)$ which will be crucial for proving the main results of this section. Throughout this section, when $R$ is not local, given a finitely generated $R$-module $M$, $\syz^n_R M$ will denote the $n$-th syzygy in a resolution of $M$ by finitely generated projective $R$-modules. 

\begin{lem}\label{1} Let $M$ be a finitely generated $R$-module. Then the following hold:
	\begin{enumerate}[\rm(1)]
		\item Let $n\le \rfd_R M$ be an integer. Then $\Tor^R_i(\syz^n_RM,T)=0$ for every $i>\rfd_RM-n$ and every module $T$ of finite flat dimension.
		
		\item Let $n\ge \rfd_R M$ be an integer. Then $\rfd_R \syz^n_R M\le 0$, i.e. $\Tor^R_i(\syz^n_R M, T)=0$ for every $i>0$ and every module $T$ of finite flat dimension. 
		
		\item Let $M,X$ be finitely generated $R$-modules such that $\rfd_R M \le 0$ and $\Ext^{>0}_R(X,M)=0$. Then $\rfd_R \Hom_R(X,M)\le 0$.   
	\end{enumerate}
\end{lem}

\begin{proof} (1): For any module $T$ of finite flat dimension  and $i>\rfd_RM-n\ge 0$ (hence $i+n>\rfd_R M$), we have $\Tor^R_i(\syz^n_R M, T)\cong \Tor^R_{i+n}(M,T)=0$.
	
	(2): For all  $i>0$, $n\ge \rfd_R M$ (hence $i+n>n\ge \rfd_R M$)  and any module $T$ of finite flat dimension, we have  $\Tor^R_i(\syz^n_R M, T)\cong \Tor^R_{i+n}(M,T)=0$.
	
	(3): Let $\p \in \spec(R)$. Let $n=\depth R_{\p}$. Since $\rfd_R M \le 0$, we have $\depth_{R_{\p}} M_{\p}\ge \depth R_{\p}$. Let $R^{\oplus b_n}\to \cdots\to R^{\oplus b_0}\to X\to 0$ be part of an $R$-free resolution of the finitely generated $R$-module $X$. Since $\Ext^{>0}_R(X,M)=0$, applying $\Hom_R(-,M)$ to the above long exact sequence, we get an exact sequence $0\to \Hom_R(X,M)\to M^{\oplus b_0}\to \cdots \to M^{\oplus b_n}\to T\to 0$ for some $R$-module $T$.
	
	By depth lemma, we have 
	\[
	\depth_{R_{\p}} \Hom_R(X,M)_{\p}\ge \inf \{\depth_{R_{\p}} M_{\p}, n+\depth_{R_{\p}}T_{\p}\}\ge \depth R_{\p}.
	\] 
	As $\depth_{R_{\p}} \Hom_R(X,M)_{\p}\ge \depth R_{\p}$ holds for all $\p\in \spec(R)$, we get $\rfd_R \Hom_R(X,M)=\sup \{\depth R_{\p}-\depth_{R_{\p}} \Hom_R(X,M)_{\p}\mid \p\in \spec(R)\}\le 0$.  
\end{proof}

\begin{rem} If $M$ is a non-zero maximal Cohen-Macaulay module over a Cohen-Macaulay ring $R$, then  it follows immediately  that $\rfd_R M=0$. More generally, see \cite[(5.3.10) Theorem]{ch}.  
\end{rem}  

 To prove descent of (SAC) and (ARC), we need to prove the following results related to vanishing of Ext and Tor and base change along finite flat dimension using techniques involving chain complexes and derived Hom and tensor product (see \cite[A.4]{ch}).

The derived category of the category of $R$-modules, denoted by $\mathcal D(R)$, is the category of $R$-complexes localized at the class of all quasi-isomorphisms (\cite[A.1.13]{ch}). The full subcategory $\mathcal D^b(R)$ consists of complexes $X$ such that $H_i(X)=0$ for all $|i|\gg 0$. Every $R$-module is identified as a complex concentrated in degree zero. 

\begin{prop}\label{3}  Let $S$ be an $R$-algebra and $M$ be an $R$-module. Then the following hold.
	\begin{enumerate}[\rm(1)]
		\item Assume $\Tor^R_{>0}(M,S)=0$ (for example, $\fd_RS<\infty$ and $\rfd_R M \le 0$). Then for all $S$-modules $N$ and for all $n\ge 0$,  $\Ext_R^n(M,N)\cong\Ext_S^n(M\otimes_RS,N)$ and $\Tor^R_n(M,N)\cong\Tor^S_n(M\otimes_RS,N)$. 
		
		\item  Let $Y$ be an $R$-module such that $\fd_R Y<\infty$. Assume $\rfd_R M \le 0$. Let $X$ be a finitely generated $R$-module. If $\Ext^{\gg 0}_R(X,M)=0$, then $\Ext^{\gg 0}_R(X,M\otimes_R Y)=0$.
		
		\item Assume $\fd_R S<\infty$ and $\rfd_R M \le 0$. Let $X$ be a finitely generated $R$-module such that $\rfd_R X \le 0$. If $\Ext^{\gg 0}_R(X,M)=0$, then $\Ext^{\gg 0}_S(X\otimes_R S,M\otimes_R S)=0$. 
		
		\item Let $Y$ be an $R$-module such that $\fd_R Y<\infty$. Assume $\rfd_R M \le 0$. Let $X$ be an $R$-module. If $\Tor_{\gg 0}^R(X,M)=0$, then $\Tor_{\gg 0}^R(X,M\otimes_R Y)=0$. 
		
		\item Assume $\fd_R S<\infty$ and $\rfd_R M \le 0$. Let $X$ be an $R$-module such that $\rfd_R X \le 0$. If $\Tor_{\gg 0}^R(X,M)=0$, then $\Tor_{\gg 0}^S(X\otimes_RS,M\otimes_R S)=0$. 
	\end{enumerate}
\end{prop}

\begin{proof} (1): The proof is the same as that of \cite[Lemma 2(ii) and (iii), page 140]{Mat} by noticing that an $R$-free resolution $L_{\bullet}\to M\to 0$ induces an $S$-free resolution $L_{\bullet}\otimes_RS\to M\otimes_RS\to 0$ as $\Tor^R_{>0}(M,S)=0$. The required isomorphism for Ext follows by computing cohomology after applying $\Hom_R(-,N)$ and using Hom-tensor adjointness (i.e. $\Hom_R(L,N)\cong \Hom_R(L,\Hom_S(S,N))\cong \Hom_S(L\otimes_R S,N)$ holds for any $R$-module $L$, see \cite[Formula 9 of Appendix A]{Mat}). The required isomorphism  for Tor follows similarly by associativity of tensor product. 
	
	(2): Since $\fd_R Y<\infty$ and $X$ is finitely generated, we have 
	$\rhom_R(X,M)\lotimes_R Y\cong \rhom_R(X,M\lotimes_R Y)$ by \cite[(A.4.23)]{ch}. 
	Since $\rfd_R M \le 0$, so $\Tor^R_{>0}(M,Y)=0$. It follows that $M\lotimes_R Y\cong M\otimes_R Y$. Hence, $\rhom_R(X,M)\lotimes_RY\cong \rhom_R(X,M\otimes_R Y)$. Since $\Ext^{\gg 0}_R(X,M)=0$, so $\rhom_R(X,M)\in \D_b(R)$, hence by \cite[(A.5.6)(ii)]{ch}, $\rhom_R(X,M)\lotimes_RY\in \D_b(R)$. Thus, $\rhom_R(X,M\otimes_R Y)\in \D_b(R)$ i.e. $\Ext^{\gg 0}_R(X,M\otimes_R Y)=0$. 
	
	(3): By (2), $\Ext^{\gg 0}_R(X,M)=0$  implies $\Ext^{\gg 0}_R(X,M\otimes_R S)=0$. 
	Since $M\otimes_R S$ is an $S$-module and $\rfd_R X \le 0$, using $(1)$ we have $\Ext^{n}_R(X,M\otimes_R S)\cong \Ext^{n}_S(X\otimes_R S,M\otimes_R S)$ for all $n\ge 0$.
	Thus $\Ext^{\gg 0}_R(X,M\otimes_R S)=0$ implies $\Ext^{\gg 0}_S(X\otimes_R S,M\otimes_R S)=0$.
	
	(4): By \cite[(A.4.20)]{ch} we have $(X\lotimes_R M)\lotimes_R Y\cong X\lotimes_R(M\lotimes_R Y)$. Since $\Tor_{\gg 0}^R(X,M)=0$, we have $X\lotimes_R M\in \D_b(R)$ and hence $(X\lotimes_R M)\lotimes_R Y\in \D_b(R)$ by \cite[(A.5.6)(ii)]{ch} (as $\fd_R Y<\infty$). Therefore $X\lotimes_R(M\lotimes_R Y) \in \D_b(R)$. Now $\rfd_R M \le 0$ implies $\Tor^R_{>0}(M,Y)=0$. Thus $M\lotimes_R Y\cong M\otimes_R Y$ and hence $X\lotimes_R(M\otimes_R Y) \in \D_b(R)$ i.e. $\Tor_{\gg 0}^R(X,M\otimes_R Y)=0$.
	
	(5): This is similar to (3); this follows by combining (4) and (1). 
\end{proof}

\begin{lem}\label{4} Let $(R,\m)\to (S,\n)$ be a homomorphism of local rings (i.e. $\m S\subseteq \n$). Let $M$ be a finitely generated $R$-module such that $\Tor^R_{>0}(M,S)=0$ (for example, $\fd_R S<\infty$ and $\rfd_R M\le 0$). If $(F_{\bullet},\partial_{\bullet})$ is a minimal $R$-free resolution of $M$, then $(F_{\bullet}\otimes_R S,\partial_{\bullet}\otimes S)$ is a minimal $S$-free resolution of $M\otimes_R S$.  
\end{lem}

\begin{proof} Clearly $F_{\bullet}\otimes_R S\to M \otimes_R S \to 0$ is exact as $\Tor^R_{>0}(M,S)=0$. Hence this is an $S$-free resolution of $M\otimes_R S$.  By minimality of $(F_{\bullet},\partial_{\bullet})$, all the entries of each $\partial_n$ is in $\m$ and hence all the entries of each $\partial_n \otimes S$ is in $\m S\subseteq \n$. Thus $F_{\bullet}\otimes_R S$ is a minimal $S$-free resolution of $M\otimes_R S$. 
\end{proof}

\begin{lem}\label{5} Let $(R,\m)\to (S,\n)$ be a homomorphism of local rings. Let $M$ be a finitely generated $R$-module such that $\Tor^R_{>0}(M,S)=0$. Then $\pd_R M<\infty$ if and only if $\pd_S (M\otimes_R S)<\infty$.  
\end{lem}

\begin{proof} 
Since $\Tor^R_{>0}(M,S)=0$, $\pd_R M<\infty$ implies the existence of a finite free $R$-resolution $(F_{\bullet},\partial_{\bullet})$ of $M$ which induces the finite free $S$-resolution $(F_{\bullet}\otimes_R S,\partial_{\bullet}\otimes S)$  of $M\otimes_R S$. Hence $\pd_S (M\otimes_R S)<\infty$.	
	
Now conversely, suppose $\pd_S(M\otimes_R S)<\infty$. Let $(R^{\oplus b_n},\partial_n)$ be a minimal $R$-free resolution of $M$. By Lemma \ref{4}, $(S^{\oplus b_n},\partial_{n}\otimes S)$ is a minimal $S$-free resolution of $M\otimes_R S$. Since $\pd_S(M\otimes_R S)<\infty$, minimality  of the above resolution implies $S^{\oplus b_n}=0$ i.e. $b_n=0$ for some $n$. Hence $\pd_R M<\infty$. 
\end{proof}

We need a variation of Proposition \ref{3}(2) and (3) for the proof of descent of (ARC). 

\begin{prop}\label{9}  Let $S$ be an $R$-algebra and $M$ be a finitely generated $R$-module. Then the following hold:
	\begin{enumerate}[\rm(1)] 
		\item  Let $Y$ be an $R$-module such that $\fd_R Y<\infty$. Suppose $\rfd_R M \le 0$. Let $X$ be a finitely generated $R$-module. If $\Ext^{>0}_R(X,M)=0$ then $\Ext^{> 0}_R(X,M\otimes_R Y)=0$ and $\Hom_R(X,M)\otimes_RY\cong \Hom_R(X,M\otimes_R Y)$.  
		
		\item Suppose $\fd_R S<\infty$ and $\rfd_R M \le 0$. Let $X$ be a finitely generated $R$-module such that $\rfd_R X \le 0$. If $\Ext^{> 0}_R(X,M)=0$ then $\Ext^{> 0}_S(X\otimes_R S,M\otimes_R S)=0$ and $\Hom_R(X,M)\otimes_RS\cong \Hom_S(X\otimes_R S,M\otimes_R S)$.  
	\end{enumerate}
\end{prop} 

\begin{proof} (1): By Lemma \ref{1}(3), we have $\rfd_R \Hom_R(X,M)\le 0$. Thus $\Hom_R(X,M)\lotimes_R Y\cong \Hom_R(X,M)\otimes_R Y$. Since $\Ext^{> 0}_R(X,M)=0$, we have $\rhom_R(X,M)\cong \Hom_R(X,M)$. Therefore $\rhom_R(X,M)\lotimes_R Y\cong \Hom_R(X,M)\otimes_R Y$. By \cite[(A.4.23)]{ch}, we get $$\Hom_R(X,M)\otimes_R Y\cong \rhom_R(X,M)\lotimes_R Y\cong \rhom_R(X,M\lotimes_RY).$$ Now $\fd_RY<\infty$ and $\rfd_R M\le 0$ imply $\Tor^R_{>0}(M,Y)=0$ and hence $M\lotimes_RY\cong M\otimes_RY$. Thus $\Hom_R(X,M)\otimes_R Y \cong \rhom_R(X,M\otimes_RY)$. Now comparing homologies of both sides, we get $\Ext^{> 0}_R(X,M\otimes_R Y)=0$ and $\Hom_R(X,M)\otimes_RY\cong \Hom_R(X,M\otimes_R Y)$. 
	
	(2): This  follows by combining (1) and the isomorphism of Ext of Proposition \ref{3}(1): 
	\begin{align*}
	&\Ext^{> 0}_R(X,M)=0\underset{\text{by(1)}}{\implies}\Ext^{> 0}_R(X,M\otimes_R S)=0\\
	\underset{\text{Proposition \ref{3}(1)}}{\implies}&\Ext^{n}_S(X\otimes_R S,M\otimes_R S)\cong \Ext^{n}_R(X,M\otimes_R S)=0 \text{ for all } n>0.
	\end{align*}
	
	Moreover, $\Hom_R(X,M)\otimes_RS\cong \Hom_R(X,M\otimes_R S)\cong \Hom_S(X\otimes_R S,M\otimes_R S)$ follows by part (1) and Proposition \ref{3}(1).   
\end{proof}    

Now we are ready to give a proof of Theorem \ref{mainsec3}.  

\begin{proof}[Proof of Theorem \ref{mainsec3}]  

(1): Let $M$ be a finitely generated $R$-module such that $\Ext^{\gg 0}_R(M,M)=\Ext^{>0}_R(M,R)=0$. By Lemma \ref*{b2.1}, it is enough to show that $\pd_R M<\infty$.
	
Let $n=\rfd_R M$. Then by Lemma \ref{1}(2), we know $\rfd_R\syz^n_R M\le 0$. 
Since $\Ext^{\gg 0}_R(M,R\oplus M)=0$, by Lemma \ref{6}(1), we have $\Ext^{\gg 0}_R(\syz^n_RM,\syz^n_R M)=0=\Ext^{\gg 0}_R(\syz^n_RM,R)$. As $\rfd_R R=0$, by Proposition \ref{3}(3) we get $\Ext^{\gg 0}_S(\syz^n_R M\otimes_R S, \syz^n_R M\otimes_R S)=0$ and $\Ext^{\gg 0}_S(\syz^n_R M\otimes_R S,S)\cong \Ext^{\gg 0}_S(\syz^n_R M\otimes_R S, R\otimes_RS)=0$. Since $S$ satisfies (SAC) and $\syz^n_R M \otimes_R S$ is a finitely generated $S$-module, by Lemma \ref*{b2.1}, we get $\pd_S(\syz^n_R M \otimes_R S)<\infty$. Therefore by Lemma \ref{5}, we have $\pd_R \syz^n_R M<\infty$, i.e. $\pd_R M<\infty$.  

(2): Let $M$ be a finitely generated $R$-module such that $\Ext^{>0}_R(M,R\oplus M)=0$ and $n=\rfd_R M$. By Lemma \ref{6}(2), we have $\Ext^{>0}_R(\syz^n_RM,\syz^n_R M)=0=\Ext^{>0}_R(\syz^n_RM,R)$. By Lemma \ref{1}(2), we know $\rfd_R\syz^n_R M\le 0$. Since $\rfd_R R=0$, by Proposition \ref{9}(2), we get $\Ext^{>0}_S(\syz^n_R M\otimes_R S, \syz^n_R M\otimes_R S)=0$ and $\Ext^{>0}_S(\syz^n_R M\otimes_R S,S)\cong \Ext^{>0}_S(\syz^n_R M\otimes_R S, R\otimes_RS)=0$.  Now $S$ satisfies (ARC) and $\syz^n_R M \otimes_R S$ is a finitely generated $S$-module. Therefore $\pd_S(\syz^n_R M \otimes_R S)<\infty$. By Lemma \ref{5}, we then get $\pd_R \syz^n_R M<\infty$, i.e. $\pd_R M<\infty$. Since $\Ext^{>0}_R(M,R)=0$, this implies $M$ is $R$-free.  
\end{proof}

The following corollary significantly generalizes \cite[Theorem 4.5(1)]{CT}, and  \cite[Theorem 2.6 (3)$\Rightarrow$(1)]{K}. 

\begin{cor} \label{cor3.8}
	Let $R$ be a local ring and $I$ be an ideal of finite projective dimension. If $R/I$ satisfies {\rm (SAC)}/{\rm (ARC)}, then $R$ satisfies {\rm (SAC)}/{\rm(ARC)}. 
\end{cor}

\begin{proof} 
 Since $R\to R/I$ is a ring homomorphism of local rings and $\fd_R R/I\le \pd_R R/I<\infty$, we get the result by Theorem \ref{mainsec3}.   
\end{proof}   

\begin{Corollary}\label{completn}
	Let $(R,\mathfrak m)$ be a  local ring. Let $\widehat{(-)}$ denote the $\fkm$-adic completion. Then, the following are equivalent: 
\begin{enumerate}[{\rm (1)}] 
\item $R$ satisfies ${\rm (SAC)}$.
\item $\widehat{R}$ satisfies ${\rm (SAC)}$.
\end{enumerate} 

\end{Corollary}	

\begin{proof}
$(2) \Rightarrow (1)$: This follows from Theorem \ref{mainsec3} as $R\to \widehat R$ is a flat extension.

$(1) \Rightarrow (2)$:  If $R$ is Artinian, then $R$ is complete and we are done. Now assume $R$ is not Artinian. Let $\m=(a_1,...,a_n)R$. Then $\widehat R \cong R[[X_1,...,X_n]]/(X_1-a_1,...,X_n-a_n)$ by \cite[Theorem 8.12]{Mat}. Now  $X_1-a_1,...,X_n-a_n$ is an $R[[X_1,...,X_n]]$-regular sequence by \cite[Lemma 5.7]{dlr}. Hence the claim follows by Corollary \ref{mainseccor3}.   
\end{proof}  

\begin{Corollary}\label{a0.4}
	All complete intersections (not necessarily local) satisfy {\rm (SAC)}.
\end{Corollary}

\begin{proof} 
Let $R$ be a complete intersection. We may assume that $R$ is local. We have $\widehat R\cong Q/(f_1,\ldots,f_c)$ for some regular local ring $Q$ of positive dimension and regular sequence $f_1,\ldots,f_c$. Since $Q$ satisfies (SAC), we have $Q/(f_1,\ldots,f_c)$ satisfies (SAC) by Corollary \ref{mainseccor3}, and hence $\widehat R$ satisfies (SAC). Now  by Theorem \ref{mainsec3} $R$ satisfies (SAC).    
\end{proof} 

\begin{cor}\label{localiz} Let $(R,\m)$ be a local ring. Then the following are equivalent:
\begin{enumerate}[{\rm (1)}] 
\item $R$ satisfies ${\rm (SAC)}$.
\item $R[X]_{(\m,X)}$  satisfies ${\rm (SAC)}$.
\end{enumerate} 
\end{cor}

\begin{proof} Since $R\to R[X]_{(\m,X)}$ is a flat map of local rings, $(2) \implies (1)$ follows by Theorem \ref{mainsec3}.  

For $(1)\implies (2)$: If $R$  satisfies (SAC) then $\widehat R$ satisfies (SAC) by Corollary \ref{completn} and hence $\widehat R[[X]]$ satisfies (SAC) by Corollary \ref{poseries}.  Therefore $\widehat {R[X]_{(\m,X)}}\cong \widehat R[[X]]$ satisfies (SAC). Hence $R[X]_{(\m,X)}$ satisfies (SAC) by Corollary \ref{completn}.  
\end{proof}

In the light of an application of Corollaries \ref{a0.3} and \ref{completn}, we note the following question.

\begin{quest}\label{q29}
Let $d\ge 2$ be an integer. 
\begin{enumerate}[\rm(1)]
\item Do all Noetherian (or Cohen-Macaulay or Gorenstein) local factorial domains of depth $d$ satisfy (SAC)?
\item Do all Noetherian (or Cohen-Macaulay or Gorenstein) local rings of depth $d$ having isolated singularity satisfy (SAC)?
 \end{enumerate}
\end{quest}

\begin{rem}\label{rem211}
If either of Question \ref{q29} has an affirmative answer for some $d$, then we can show that all Noetherian (or Cohen-Macaulay or Gorenstein)  local rings containing a field satisfy (SAC). 

Indeed, let $d$ be the integer in Question \ref{q29}. Let $R$ be a Noetherian local ring containing a field. Then, so is $R/Q$ where $Q$ is an ideal generated by a maximal regular sequence. 
We consider the completion of $(R/Q)[[x_1, \dots, x_d]]$. Then, the completion has depth $d$. By Heitmann's theorems (\cite[Theorem 8]{heit}, \cite[Main Theorem]{heit2}), there exists a Noetherian local factorial domain $S$ (or a ring having isolated singularity) such that $\widehat{S}$ becomes the completion of $(R/Q)[[x_1, \dots, x_d]]$. Hence, if Question \ref{q29} has an affirmative answer, $S$ satisfies (SAC). By Corollary \ref{completn}, $\widehat{S}$ satisfies (SAC). It follows that so does $(R/Q)[[x_1, \dots, x_d]]$ by Corollary \ref{completn}, and so does $R$ by Corollaries \ref{a0.3} and \ref{poseries}. 

In addition, we first assume that $R$ is Cohen-Macaulay or Gorenstein, then so are all rings appearing in the above argument. 
\end{rem}

\begin{rem}
We don't know if the above argument works for (ARC) since we don't know if (ARC) of $R$ is inherited by that of $R/xR$, where $x$ is a non-zerodivisor, without assuming the completeness of $R$ (see, for example, \cite[Proposition 2.2]{K}).

On the other hand, very recently, it is proved that all Cohen-Macaulay local factorial domains satisfy (ARC) (\cite{KOT}). 
\end{rem}

\section{Transfer of (SAC) along $R \to R/I^\ell$} \label{section4}  

Let $(R, \fkm) \to (S, \fkn)$ be a local homomorphism of local rings. In the previous section, we prove that (SAC) for $S$ implies (SAC) for $R$ if the flat dimension of $S$ over $R$ is finite. Here we explore other special cases by using the notion of Ulrich ideals. See Definition \ref{Ulrich} to recall the definition of Ulrich ideals. We note that all parameter ideals are Ulrich ideals. 

\begin{fact} {\rm (\cite[Lemma 2.4]{K})}\label{c2.7}
Let $(R, \fkm)$ be a  local ring and $I$ an $\fkm$-primary ideal of $R$. Then we have the following.
\begin{enumerate}[{\rm (1)}] 
\item  Suppose that $R/I$ is a Gorenstein ring. If $\Ext_R^{>0} (M, R/I)=0$, then $\Tor_{>0}^R (M, R/I)=0$.
\item Suppose that $N$ is a finitely generated $R/I$-module. If $\Tor_{>0}^R (M, R/I)=0$, then $\Ext_R^{i} (M, N)\cong\Ext_{R/I}^{i} (M/IM, N)$ for all $i\in \mathbb{Z}$.
\end{enumerate}
\end{fact}

The following lemma plays an important role in the proof of Theorem \ref{ul}. The idea of this proof is similar to \cite[Theorem 2.6 (1) $\Rightarrow$ (4)]{K}, but we include a proof since some details are different. 

\begin{lem}\label{estimate} Let $(R,\m)$ be a  local ring and $n\ge 2$ an integer. Let $Q_n\subseteq Q_{n-1}\subseteq \cdots \subseteq Q_1$ be a sequence of $\m$-primary ideals. For each $n\ge i\ge 2$, we assume that
\begin{center}
$Q_{i-1}/Q_i$ is a free $R/Q_1$-module of rank $a_i$ \quad and \quad $\frac{(1+\sum_{i=2}^{\ell-1}a_i)}{a_\ell}<1$ for all $2\le \ell \le n$. 
\end{center}
Then, if $R/Q_1$ is Gorenstein and satisfies {\rm (SAC)}, $R/Q_\ell$ also satisfies {\rm (SAC)} for all $2\le \ell\le n$.   
\end{lem}

\begin{proof}
Set $R_i=R/Q_i$ for $i>0$. We then have the exact sequences
\begin{align}\label{xy1}
0\to Q_{i-1}/Q_{i}\to R_i\to R_{i-1}\to 0 
\end{align}
of $R$-modules for all $2\le i \le \ell$. By assumption, $Q_{i-1}/Q_{i}\cong R_1^{\oplus a_i}$ for all $ 2\le i\le n$.   

Let $2\le \ell \le n$. Suppose that $M$ is a finitely generated $R_\ell$-module such that $\Ext_{R_\ell}^{>0}(M, R_\ell)=\Ext_{R_\ell}^{\gg0}(M, M)=0$. We aim to show that $\Ext^{>0}_{R_\ell}(M,M)=0$. By applying the functor $\Hom_{R_\ell} (M, -)$ to (\ref{xy1}), we get a long exact sequences and isomorphisms
\begin{align}\label{xy2}
\begin{split}
\cdots \to &\Ext_{R_\ell}^{j} (M, R_1)^{\oplus a_i} \to \Ext_{R_\ell}^{j} (M, R_i) \to \Ext_{R_\ell}^{j} (M, R_{i-1})\\
 \to & \Ext_{R_\ell}^{j+1} (M, R_1)^{\oplus a_i} \to \cdots\\ 
 \text{and}\\
 &\Ext_{R_\ell}^{j} (M, R_{\ell-1})\cong \Ext_{R_\ell}^{j+1} (M, R_1)^{\oplus a_\ell} 
\end{split}
\end{align}
of $R_\ell$-modules for all $2\le i\le \ell-1$ and $j>0$.
Set $E_j=\ell_{R_\ell} (\Ext_{R_\ell}^{j} (M, R_1))$ for $j>0$. By (\ref{xy2}), 
\[
\begin{split}
E_{j+1}{\cdot}a_\ell&=\ell_{R_\ell} (\Ext_{R_\ell}^{j} (M, R_{\ell-1}))\\
&\le E_j{\cdot}a_{\ell-1} + \ell_{R_\ell} (\Ext_{R_\ell}^{j} (M, R_{\ell-2}))\\
&\le E_j{\cdot}\left( a_{\ell-1} + a_{\ell-2}\right) + \ell_{R_\ell} (\Ext_{R_\ell}^{j} (M, R_{\ell-3})) \le \cdots \\
& \le E_j{\cdot}\left(1+\sum_{i=2}^{\ell-1}a_i\right),
\end{split}
\]
where the first equality follows from the isomorphism of \eqref{xy2} and the first inequality follows from the long exact sequence of \eqref{xy2} for $i=\ell-1$ and the fact that for $X\xrightarrow{f}Y\xrightarrow{g}Z$, $\ell_R(Y)\le \ell_R(X)+\ell_R(Z)$. The other inequalities follow in the same way as the first inequality by using the long exact sequence of \eqref{xy2} for $i=\ell-2, \dots, 2$ recursively.

Therefore, we obtain that
\[
E_{j+1}\le \frac{(1+\sum_{i=2}^{\ell-1}a_i)}{a_\ell}{\cdot}{E_j}
\]
for all $j>0$. It follows that $E_{m+1}\le \left(\frac{(1+\sum_{i=2}^{\ell-1}a_i)}{a_\ell}\right)^m E_1$ for $m\gg 0$. Since $\frac{(1+\sum_{i=2}^{\ell-1}a_i)}{a_\ell}<1$ by assumption, we obtain that $E_{m+1}=0$, that is, $\Ext_{R_\ell}^{m+1}(M, R_1)=0$ for all $m\gg 0$.  

On the other hand, assume that $\Ext_{R_\ell}^{j+1}(M, R_1)=0$ for some $j>0$. By the isomorphism of (\ref{xy2}), we get $\Ext_{R_\ell}^{j} (M, R_{\ell-1})=0$. It follows that $\Ext_{R_\ell}^{j} (M, R_{\ell-2})=0$ by the long exact sequence of (\ref{xy2}) where $i=\ell-1$. By looking at the long exact sequence of (\ref{xy2}), where $i=\ell-2, \dots, 2$ recursively,  we finally obtain $\Ext_{R_\ell}^{j} (M, R_{1})=0$. Thus $\Ext_{R_\ell}^{j+1}(M, R_1)=0$ implies that $\Ext_{R_\ell}^{j}(M, R_1)=0$ for all $j>0$. Hence $\Ext_{R_\ell}^{m+1}(M, R_1)=0$ for all $m\gg 0$ concludes that $\Ext_{R_\ell}^{>0}(M, R_1)=0$.
We have $\Tor_{>0}^{R_\ell} (M, R_1)=0$ by Fact \ref{c2.7} (1). It follows that 
\begin{align}\label{xy3}
\Tor_{>0}^{R_\ell} (M, R_i)=0 \ \ \text{ for all $1\le i \le \ell-1$}
\end{align}
by applying the functor $M\otimes_{R_\ell} -$ to the exact sequence (\ref{xy1}). We furthermore have 
\begin{align}\label{xy4}
0\to (M/Q_1M)^{\oplus a_i} \to M/Q_i M \to M/Q_{i-1} M \to 0 
\end{align}
of $R_\ell$-modules for all $2\le i \le \ell$ by (\ref{xy1}) and (\ref{xy3}).  
Therefore, by applying the functor $\Hom_{R_\ell} (M, -)$ to (\ref{xy4}), we get 
\begin{align}\label{xy4.5}
\begin{split}
\cdots \to& \Ext_{R_\ell}^{j} (M, M/Q_1M)^{\oplus a_i} \to \Ext_{R_\ell}^{j} (M, M/Q_{i}M) \to \Ext_{R_\ell}^{j} (M, M/Q_{i-1}M) \\
\to& \Ext_{R_\ell}^{j+1} (M, M/Q_1M)^{\oplus a_i} \to \cdots\\ 
&\text{and}\\
&\Ext_{R_\ell}^{j} (M, M/Q_{\ell-1}M)\cong \Ext_{R_\ell}^{j+1} (M, M/Q_1M)^{\oplus a_\ell} 
\end{split}
\end{align}
of $R_\ell$-modules for all $2\le i\le \ell-1$ and $j\gg 0$. 
Set $E'_j=\ell_{R_\ell} (\Ext_{R_\ell}^{j} (M, M/Q_1M))$ for $j>0$. Then by the same calculation as above, we obtain that $\Ext_{R_\ell}^{\gg0}(M, M/Q_1M)=0$. It follows that $\Ext_{R_1}^{\gg 0}(M/Q_1M, M/Q_1M)=0$ by (\ref{xy3}) and Fact \ref{c2.7} (2) (here our ring is $R_\ell=R/Q_\ell$ and the ideal is $Q_1/Q_\ell$). Since $R_1$ is an Artinian Gorenstein ring by assumption, we have $\Ext_{R_1}^{>0}(M/Q_1M, R_1)=0$. Therefore, by the assumption that $R_1$ satisfies {\rm (SAC)}, we obtain that $\Ext_{R_1}^{>0}(M/Q_1M, M/Q_1M)=0$.
By Fact \ref{c2.7} (2), it follows that $\Ext_{R_\ell}^{>0}(M, M/Q_1M)=0$ . Applying the functor $\Hom_{R_\ell}(M, -)$ to (\ref{xy4}) for each $2\le i\le \ell$, we get that $\Ext^{>0}_{R_\ell}(M,M/Q_iM)=0$ for each $2\le i\le \ell$. In particular,  $\Ext^{>0}_{R_\ell}(M,M)=0$ since $M/Q_\ell M=M$ (as $M$ is an $R/Q_\ell$-module). 
\end{proof} 

We can apply Lemma \ref{estimate} to Ulrich ideals.

\begin{lem}\label{ulrank} 
Let $(R,\m)$ be a Cohen-Macaulay local ring of dimension $n$ and $I$ an Ulrich ideal of $R$. If $\mu(I)=c$, then $I^{i}/I^{i+1}\cong (R/I)^{\oplus \left[ \binom{i+n-1}{n-1}+(c-n)\binom{i-1+n-1}{n-1}\right]}$ for every $i\ge 1$.
\end{lem}

\begin{proof} 
We first prove that $I/Q\cong (R/I)^{\oplus (c-n)}$. Indeed, by noting that $I^2=Q I$, we have the exact sequence $0 \to Q/QI =Q/I^2 \to I/I^2 \to I/Q \to 0$. We also observe that $Q/QI \cong (Q/Q^2) \otimes_R R/I \cong (R/I)^{\oplus n}$ and $I/I^2\cong (R/I)^{\oplus c}$. By Auslander-Buchsbaum formula (\cite[Theorem 1.3.3]{BH}), it follows that $I/Q$ is an $R/I$ free module and the rank is $c-n$. 

We second prove that $I^i/Q^i\cong (R/I)^{\oplus (c-n)\binom{i-1+n-1}{n-1}}$ for all $i>0$. We observe the exact sequence 
\[
0 \to I^i/Q^i(=Q^{i-1}I/Q^i) \to Q^{i-1}/Q^i \to Q^{i-1}/Q^{i-1}I \to 0,
\] 
where $I^i/Q^i=Q^{i-1}I/Q^i$ follows from $I^2=Q I$ for $i\ge 2$. We then  obtain that $I^i/Q^i \cong (I/Q)^{\oplus \binom{i-1+n-1}{n-1}}$ because \begin{center}
$Q^{i-1}/Q^i\cong (R/Q)^{\oplus \binom{i-1+n-1}{n-1}}$ and $Q^{i-1}/Q^{i-1}I \cong Q^{i-1}/Q^i \otimes_R R/I \cong (R/I)^{\oplus \binom{i-1+n-1}{n-1}}$.
\end{center} 
Since $I/Q\cong (R/I)^{\oplus (c-n)}$ by the first step, it follows that $I^i/Q^i\cong (R/I)^{\oplus (c-n)\binom{i-1+n-1}{n-1}}$ for all $i>0$.

Now we prove the assertion. Consider the following exact sequence:
\[
0 \to Q^i/Q^i I(=Q^i/I^{i+1}) \to I^i/I^{i+1} \to I^i/Q^i \to 0.
\]
Note that $Q^i/Q^iI=Q^i/I^{i+1} \cong (R/I)^{\oplus \binom{i+n-1}{n-1}}$ and $I^i/Q^i\cong (R/I)^{\oplus (c-n)\binom{i-1+n-1}{n-1}}$. Therefore, by Auslander-Buchsbaum formula, we obtain that $I^i/I^{i+1}$ is an $R/I$-free module of rank $\binom{i+n-1}{n-1} + (c-n)\binom{i-1+n-1}{n-1}$. 
\end{proof}

	\begin{lem}\label{34}  Let $M,N$ be  $R$-modules. Let $\{R_i\}_{i=1}^\ell$ be a sequence of $R$-algebras such that there is an exact sequence 
\begin{align}\label{eq451}
0\to R_{1}^{\oplus a_i}\to R_i\to R_{i-1}^{\oplus b_i}\to 0
\end{align}
of $R$-modules, where $a_i, b_i$ are integers, for each $2\le i\le \ell$. (e.g., $R_i:=R/I^i$ and $0 \to I^{i-1}/I^i(\cong (R/I)^{\oplus a_i}) \to R_i \to R_{i-1}\to 0$ where $I$ is an Ulrich ideal.) Suppose $\Tor^R_{>0}(M\oplus N,R_1)=0$. Then the following hold.
		\begin{enumerate}[\rm(1)]
	\item For all $1\le i\le \ell$, $\Tor^R_{> 0}(M\oplus N,R_i)=0$.
	\item If  $\Ext^{>0}_R(M,N\otimes_R R_1)=0$, then $\Ext^{>0}_{R_i}(M\otimes_RR_i,N\otimes_RR_i)=0$ holds for each $1\le i\le \ell$.	
	\item If $\Ext^{\gg 0}_R(M,N\otimes_R R_1)=0$, then $\Ext^{\gg 0}_{R_i}(M\otimes_RR_i,N\otimes_RR_i)=0$ holds for each $1\le i\le \ell$. 	
\end{enumerate}
	\end{lem}  
	
	\begin{proof} 
$(1)$: The case of $i=1$ follows by hypothesis. For $i=2$, applying $(M\oplus N)\otimes -$ to \eqref{eq451}, we get $\Tor^R_{> 0}(M\oplus N,R_2)=0$. Suppose the claim for Tor vanishing has been proved for $i-1$. Then again applying $(M\oplus N)\otimes -$ to \eqref{eq451} and the assumption, we get $\Tor^R_{> 0}(M\oplus N,R_i)=0$. 
		
$(2)$: We first note that the isomorphism $\Ext^n_{R_i}(M\otimes_R R_i,N\otimes_R R_i)\cong \Ext^{n}_{R}(M,N\otimes_RR_i)$ holds for all $n$ and for all $0\le i\le \ell$ by Proposition \ref{3}(1) (since we have $\Tor^R_{>0}(M,R_i)=0$ for all $1\le i\le \ell$ by (1)). Hence, it is enough to prove $\Ext^{> 0}_{R}(M,N\otimes_RR_i)=0$ holds for each $1\le i\le \ell$.
 
The case of $i=1$ follows by hypothesis. Since $\Tor^R_{>0}(N,R_i)=0$ for all $1\le i\le \ell$ by (1), for each $2\le i\le \ell$, we get an exact sequence 
\[
0\to (N\otimes_R R_{1})^{\oplus a_i}\to N\otimes_R R_i\to (N\otimes_R R_{i-1})^{\oplus b_i}\to 0
\] 
of $R$-modules. If the claim $\Ext^{> 0}_{R}(M,N\otimes_RR_{i-1})=0$ holds, then again applying $\Hom_R(M,-)$ to the above exact sequence, we get $\Ext^{>0}_{R}(M,N\otimes_RR_i)=0$.  
		
$(3)$: The proof is similar to $(2)$.  
\end{proof}

\begin{Theorem}\label{ul} Let $R$ be a Cohen-Macaulay local ring of dimension $n$. Let $I$ be an Ulrich ideal such that $R/I$ is Gorenstein. Then the following hold true.
\begin{enumerate}[\rm(1)] 
\item If $R/I$ satisfies {\rm (SAC)} and $n \ge 2$, then $R/I^\ell$ satisfy {\rm (SAC)} for all $2\le \ell \le n$. 
\item If $R/I^\ell$ satisfies {\rm (SAC)} for some $\ell>0$ and $\mu(I)-\dim R>1$, then $R$ satisfies {\rm (SAC)}.
\end{enumerate}
\end{Theorem} 

\begin{proof} 
(1): We apply Lemma \ref{estimate} with $Q_i :=I^i$ for $1\le i\le n$. By Lemma \ref{ulrank}, $I^{i-1}/I^{i}$ is $R/I$-free and thus we only need to check that the rank $a_i$ of $I^{i-1}/I^{i}$ satisfies the condition  $\frac{(1+\sum_{i=2}^{\ell-1}a_i)}{a_\ell}<1$ for each $2\le \ell \le n$.

Let $\mu(I)=c$. By Lemma \ref{ulrank} we have $a_i=\binom{i-1+n-1}{n-1}+(c-n)\binom{i-2+n-1}{n-1}$. If $\ell=2$, then $\frac{(1+\sum_{i=2}^{\ell-1}a_i)}{a_\ell}=\dfrac 1{a_2}=\dfrac{1}{\binom{2-1+n-1}{n-1}+(c-n)\binom{2-2+n-1}{n-1}}=\dfrac{1}{n+(c-n)}=\dfrac 1 c<1$ as $c\ge \dim R=n>1$. If $\ell\ge 3$, then 
\begin{eqnarray*}1+\sum_{i=2}^{\ell-1} a_i&=&1+\sum_{i=2}^{\ell-1}\binom{i-1+n-1}{n-1}+(c-n)\binom{i-2+n-1}{n-1} \\&=&\sum_{i=0}^{\ell-2}\binom{i+n-1}{n-1}+(c-n)\sum_{i=0}^{\ell-3}\binom{i+n-1}{n-1}\\&=&\binom{\ell-2+n}{n}+(c-n)\binom{\ell-3+n}{n}\\&<&\binom{\ell-2+n}{n-1}+(c-n)\binom{\ell+n-3}{n-1}=a_\ell 
\end{eqnarray*}
(as $\ell \le n$ implies $\binom{\ell-2+n}{n}<\binom{\ell-2+n}{n-1}$ and $\binom{\ell-3+n}{n}\le \binom{\ell-3+n}{n-1}$). Thus we get $\frac{(1+\sum_{i=2}^{\ell-1}a_i)}{a_\ell}<1$ for all $2\le \ell \le n$. Therefore, the result follows from Lemma \ref{estimate}. 

(2): The idea of the proof is similar to \cite[Theorem 2.6 (3) $\Rightarrow$ (1)]{K}. We include the proof since some details are different. Let $M$ be a finitely generated $R$-module such that $\Ext^{\gg 0}_R(M,M)=\Ext^{>0}_R(M,R)=0$. We claim that $M$ is free. By Lemma \ref{6}, we have $\Ext^{\gg 0}_R(\syz^n_R M,\syz^n_R M)=0=\Ext^{>0}_R(M,R)$ for all $n\ge 1$. Therefore, passing to high enough syzygy, we may assume $M$ is a MCM $R$-module.  

Let $Q$ be a parameter ideal such that $I^2=QI$. Note that $\Ext^{>0}_R(M,R)=0$ implies $\Ext^{>0}_R(M,R/Q)=0$. Let $a:=\mu(I)-\dim R$. By the proof of Lemma \ref{ulrank} we have $I/Q\cong(R/I)^{\oplus a}$. Applying $\Hom_R(M,-)$ to the exact sequence $0\to (R/I)^{\oplus a} \cong I/Q\to R/Q\to R/I\to 0$, we get $\Ext^i_R(M,R/I)\cong \Ext^{i+1}_R(M,R/I)^{\oplus a}$  for all $i\ge 1$. Let $n_i=\ell_R (\Ext^i_R(M,R/I))$. Then, $n_i=an_{i+1}$  for all $ i\ge 1$ and $n_1=a^in_{i+1}$  for all $i\ge 1$. Since $a>1$, we have $n_{i+1}=\dfrac{n_1}{a^i}<1$  for all $i\gg 0$. Since $n_i$'s are integers for all $i\gg 0$, we get $n_i=0$. Hence  $\Ext^i_R(M,R/I)=0$  for all $ i\gg 0$. Now $\Ext^j_R(M,R/I)\cong \Ext^{j+1}_R(M,R/I)^{\oplus a}$  for all $j\ge 1$ imply that $\Ext^j_R(M,R/I)=0$  for all $ j\ge 1$. By Fact \ref{c2.7}(1), we have $\Tor^R_{>0}(M,R/I)=0$. 
Since $M$ is MCM and  $\Ext^{\gg 0}_R(M,M)=0$, we have $\Ext^{\gg 0}_R(M,M/QM)=0$. Applying $-\otimes_R M$ to $0\to (R/I)^{\oplus a} \to R/Q\to R/I\to 0$, we get an exact sequence $0\to (M/IM)^{\oplus a}\to M/QM \to M/IM \to 0$. Applying $\Hom_R(M,-)$ to this we get $\Ext^i_R(M,M/IM)\cong \Ext^{i+1}_R(M,M/IM)^{\oplus a}$ for all $i\gg 0$. Again considering length of each $\Ext^i_R(M,M/IM)$, similar to the above argument, we get $\Ext^j_R(M,M/IM)=0$ for all $j\gg 0$. 

Now we have $\Tor^R_{>0}(M,R/I)=0=\Ext^{\gg 0}_R(M,M\otimes_R R/I)=\Ext^{>0}_R(M,R/I)$ and exact sequences $0\to (R/I)^{\oplus a_i}\cong I^{i-1}/I^{i}\to R/I^i\to R/I^{i-1}\to 0$ for all $i\ge2$. Applying Lemma \ref{34} with $M=N$ and $N=R$ respectively and $R_i:=R/I^i$ for all $i\ge 1$, we get $\Ext^{\gg 0}_{R/I^\ell}(M\otimes_R R/I^\ell, M\otimes_R R/I^\ell)=0=\Ext^{>0}_{R/I^\ell}(M\otimes_R R/I^\ell, R/I^\ell)$. Now $R/I^\ell$ satisfies (SAC) implies $\pd_{R/I^\ell}(M\otimes_R R/I^\ell)$ is finite. Since $\Tor^{>0}_R(M,R/I^\ell)=0$, by Lemma \ref{5}, we get $\pd_R M<\infty$. Now \cite[Lemma 1 (iii), page 154]{Mat} and $\Ext^{>0}_R(M,R)=0$ imply $M$ is free.
\end{proof}


We note assumptions on Ulrich ideal $I$ in Theorem \ref{ul}.

	\begin{rem}\label{ulag}
Let $R$ be a Cohen-Macaulay local ring and $I$ an Ulrich ideal of $R$. 
Then, $R$ is Gorenstein if and only if $R/I$ is Gorenstein and $\mu_R(I)=\dim R+1$ (\cite[Corollary 2.6(b)]{GOTWY}). On the other hand, if $R$ is non-Gorenstein almost Gorenstein rings and the type of $R$ is a prime number (for example of such rings, see \cite[Theorem 3.4]{GOTWY}), then $R/I$ is Gorenstein and $\mu_R(I)-\dim R>1$ (\cite[Corollary 2.11]{GTT2}). Other examples of Ulrich ideals $I$ such that $R/I$ is Gorenstein and $\mu_R(I)-\dim R>1$ are in \cite[Example 3.11]{K}.
\end{rem}

The following is an example of a ring satisfying the assumption of Theorem \ref{ul}(2).

\begin{ex}\label{ex47}
Let $R=k[[t^8, t^{11}, t^{12}, t^{14}, t^{18}]]$ be a numerical semigroup ring over a field $k$, where $k[[t]]$ is the formal power series ring over $k$.  Let $I=(t^8, t^{12}, t^{14}, t^{18})$ be an ideal of $R$ and $Q=(t^8)$ a reduction of $I$. Then $I$ is an Ulrich ideal such that $R/I$ is a complete intersection. Indeed, one can check that $I^2=QI$, $\ell_R(R/I)=2$, and $\ell_R(I/I^2)=8=\mu_R(I) \ell_R(R/I)$. It follows that $I/I^2\cong (R/I)^{\oplus 4}$.
Noting that $R/I$ satisfies (SAC) since $R/I$ is a hypersurface, $R$ satisfies (SAC).
\end{ex}


\begin{Theorem}\label{a0.5}
Let $(R, \fkm)$ be a Gorenstein local ring and $x_1, \dots, x_{n}$ be a regular sequence on $R$. Set $Q=(x_1, \dots, x_n)$. Then the following conditions are equivalent:
\begin{enumerate}[{\rm (1)}] 
\item $R$ satisfies {\rm (SAC)}.
\item $R/Q$ satisfies {\rm (SAC)}.
\item $R/Q^\ell$ satisfies {\rm (SAC)} for some integer $\ell>0$.
\item $R/Q^\ell$ satisfies {\rm (SAC)} for all integers $1\le \ell\le n$.
\end{enumerate}
\end{Theorem}

%


\begin{proof}

$(1) \Leftrightarrow (2)$ follows from Corollary \ref{a0.3}. Since $\pd_R R/Q^\ell$ is finite by \cite{EN}, $(3)\Rightarrow (1)$ follows from Corollary \ref{cor3.8}. $(4) \Rightarrow (3)$ is clear. Hence, it is enough to show that $(1)\Rightarrow (4)$. 

$(1)\Rightarrow (4)$: We may assume that $n=\dim R$, that is, $Q$ is a parameter ideal of $R$. Indeed, if $n< \dim R$, we can choose $x_{n+1}\in \fkm$ so that $x_1, \dots, x_n, x_{n+1}$ is a regular sequence of $R$. Then $x_{n+1}$ is a regular element of $R/Q^\ell$ for all $\ell>0$ since $R/Q^\ell$ is a Cohen-Macaulay ring of dimension $\dim R-n$.
Hence, for all $\ell>0$, $R/Q^\ell$ satisfies {\rm (SAC)} if and only if $R/\left[Q^\ell+(x_{n+1})\right]$ satisfies {\rm (SAC)} by Corollary \ref{a0.3}.

Suppose that $n=\dim R$. By Corollary \ref{a0.3} again, we may assume that $n\ge 2$ and $\ell\ge 2$. So now the claim follows from Theorem \ref{ul}(1) as every parameter ideal is an Ulrich ideal.   
\end{proof}

The following assertion replaces (ARC) with (SAC) in the assertion of \cite[Theorem 1.3(3)]{K}. Therefore, since (SAC) implies (ARC), this recovers \cite[Theorem 1.3(3)]{K}.

\begin{Corollary}\label{c2.11}
Let $R$ be a Cohen-Macaulay local ring. Suppose that there exists an Ulrich ideal $I$ such that $R/I$ is a complete intersection. Then $R$ satisfies {\rm (SAC)}.
\end{Corollary}

\begin{proof}
By \cite[Theorem 1.3(2)]{K}, there exists an isomorphism of rings $R/\fkq\cong S/Q^2$ for some parameter ideal $\fkq$ of $R$, local complete intersection $S$ of dimension $r$, and parameter ideal $Q$ of $S$. Therefore, since $S$ satisfies (SAC) by Corollary \ref{a0.4}, we obtain the assertion by applying Theorem \ref{a0.5}.
\end{proof}
	

For examples of rings that have an Ulrich ideal whose residue ring is a complete intersection, see \cite[Example 3.11]{K}. 

Due to Theorem \ref{a0.5}, we can strengthen \cite[Proposition 2.11 and Theorem 2.16]{K} by replacing (ARC) with (SAC) in the conclusion. We should note that we need not assume the completeness of a given ring, although \cite[Proposition 2.11]{K} assumes it. (This is because of the difference of the hypothesis between Theorem \ref{a0.3} and \cite[Proposition 2.2]{K}.)

\begin{cor}\label{c411}
Let $s$, $t$ be positive integers and assume that $2s\le t+1$. Let $\alpha_{ij}$ be positive integers for all $1\le i\le s$ and $1\le j\le t$.
Suppose that $S$ is a Gorenstein local ring and $\{ x_{ij} \}_{1\le i\le s, 1\le j\le t}$ forms a regular sequence on $S$.
Let $I=I_s (x_{ij}^{\alpha_{ij}})$ be an ideal of $S$ generated by $s\times s$ minors of the $s\times t$ matrix $(x_{ij}^{\alpha_{ij}})$.
Set $R=S/I$. 
Then $S$ satisfies {\rm (SAC)} if and only $R$ satisfies {\rm (SAC)}.
\end{cor}

\begin{proof}
This can be proved in the same way of the proof of \cite[Proposition 2.11]{K}, just by replacing with (ARC) with (SAC) and using Theorem \ref{a0.5}.
\end{proof}

We next investigate determinantal rings, which are not local. Let $A$ be a commutative ring and 
\[
A[\mathbf{X}]=A[X_{ij}\mid 1\le i\le s, 1\le j\le t]
\] 
be a polynomial ring over $A$. 
Suppose that $s\le t$ and $I_s (\mathbf{X})$ is an ideal of $A[\mathbf{X}]$ generated by $s\times s$ minors of the $s\times t$ matrix $\mathbf{X}=(X_{ij})$. Then $A[{\bf X}]/I_s({\bf X})$ is called the {\it determinantal ring} over $A$. 


\begin{cor}\label{det52}
Suppose that $A$ is a complete intersection for all localizations at maximal ideals. Then the determinantal ring $A[\mathbf{X}]/I_s(\mathbf{X})$ also satisfies {\rm (SAC)} if $2s\le t+1$.
\end{cor}

\begin{proof}
This can be proved in the same way of the proof of \cite[Theorem 2.16]{K}, just by replacing with (ARC) with (SAC) and using Theorem \ref{a0.5} .
\end{proof}

\section{Gluing of numerical semigroups}\label{sectionglue}

In this section, we provides a new class of rings satisfying (SAC) via the gluing of numerical semigroup rings. For the notion of gluing of numerical semigroups, one can consult, for example,  \cite[p.130]{rg}.
We use the following lemma.  

\begin{lem} \label{l47}
Let $(A, \fkm, k)$ be a local ring. Let $n\ge 2$ be an integer and $a\in \fkm$. Set 
\[
R:=A[X]/(X^n-a),
\] 
where $A[X]$ denotes the polynomial ring over $A$. We denote by $x$ the image of $X$ in $R$. Then we have the following.
\begin{enumerate}[\rm(1)] 
\item $R$ is a local ring with the maximal ideal $\fkm R + xR$ and a free $A$-module of rank $n$. \item Suppose that $A$ is a Cohen-Macaulay ring and $I$ is an Ulrich ideal of $A$. Then $IR$ is also an Ulrich ideal of $R$. 
\item If $A$ is a Cohen-Macaulay ring and $\m$ is an Ulrich ideal of $A$ (equivalently, $A$ has minimal multiplicity when $k$ is infinite), then $\m R$ is an Ulrich ideal of $R$ such that $R/\fkm R$ is a hypersurface. 
\end{enumerate}  

\end{lem}

\begin{proof} 
(1) and (2) are possibly known, but we include a proof for the sake of completeness. 

(1): 
We claim that $\{1,x,\cdots,x^{n-1}\}$ is a free basis of $R$ as an $A$-module. It is clear that $R=A+\sum_{i=1}^{n-1} Ax^i$. Assume that $b_0+\sum_{i=1}^{n-1}b_ix^i=0$ in $R$ for some $b_0,b_1,\dots,b_{n-1}\in A$. Then, $b_0+\sum_{i=1}^{n-1}b_iX^i\in (X^n-a)A[X]$ in $A[X]$. 
It follows that $b_0+\sum_{i=1}^{n-1}b_iX^i=g(X)(X^n-a)$ for some $g(X)\in A[X]$. If at least one $b_i$ is non-zero, then the left hand side is a non-zero polynomial in $A[X]$ of degree $\le n-1$.  Hence $g(X)\ne 0$. But this leads a contradiction that the degree of the polynomial of the left hand side is $\deg(g(X)(X^n-a))\ge n$. Therefore we must have $b_i=0$ for all $0\le i\le n-1$. Thus $\{1,x,\cdots,x^{n-1}\}$ is a free basis of $R$ over $A$. 

Now let $\p$ be a maximal ideal of $R$. Since $A\hookrightarrow R$ is integral, $\m=\p \cap A$. Hence $\m R \subseteq \p$. In addition, since $x^n=a\in \fkm R\subseteq \p$, we have $x\in \p$. It follows that $xR+\m R\subseteq \p$. Since $xR+\m R$ is a maximal ideal of $R$, we have $\p=\m R+xR$ and hence $R$ is local.  

(2): This is straightforward and follows by an argument similar to \cite[Lemma 4.1]{GTT2}. Indeed, let $I$ be an Ulrich ideal of $A$ with a reduction $Q$. Then $I^2R=QIR$, and $IR/I^2R$ is a free $R/IR$-module. 

(3): By (2), $\m R$ is an Ulrich ideal of $R$. On the other hand, $R/\m R$ is an Artinian ring with the maximal ideal $(\fkm R + xR)/\fkm R$. It follows that $R/\m R$ is a hypersurface.  
\if   
We have natural map $A \xrightarrow{a\mapsto a+(X^n-a)} A[X]/(X^n-a)$; to see this is injectie, we need to show $A\cap (X^n-a)=0$ in $A[X]$. So let $b\in A \cap (X^n-a)$. Then, $b=f(X)(X^n-a)$ for some $f(X)\in A[X]$. Now if $b\ne 0$, then $f(X)\ne 0$, then since $X^n-a$ is a monic polynomial, so $0=\deg(b)=deg (f(X)(X^n-a))=n+\deg f(X)$,  contradicting $n\ge 2$. Thus we must have $b=0$. 
This establishes injectivity of the natural map $A \xrightarrow{a\mapsto a+(X^n-a)} A[X]/(X^n-a)$. Now we claim that $\{1,x,\cdots,x^{n-1}\}$ is a free basis of $R$ as an $A$-module, where $x$ is the image of $X$ in $R$. Since $x^n=a$ in $R$, so  clearly $R=A+\sum_{i=1}^{n-1} Ax^i$ holds. Now let if possible $b_0+\sum_{i=1}^{n-1}b_ix^i=0$ in $R$ for some $b_0,b_1,\cdots,b_{n-1}\in A$. Then, $b_0+\sum_{i=1}^{n-1}b_iX^i\in (X^n-a)A[X]$ in $A[X]$. So, $b_0+\sum_{i=1}^{n-1}b_iX^i=g(X)(X^n-a)$ for some $g(X)\in A[X]$. If at least one $b_i$ is non-zero, then the left hand side is a non-zero polynomial in $A[X]$ of degree $\le n-1$, hence we would also have $g(X)\ne 0$, but then $\deg(g(X)(X^n-a))\ge n$, contradiction! Thus we must have $b_i=0$ for all $0\le i\le n-1$. Thus $\{1,x,\cdots,x^{n-1}\}$ is a free basis of $R$ over $A$.  So $A \to R$ is a flat map with $\m R \subseteq \mathfrak n$. 

Let $\p$ be a maximal ideal of $R$. Since $A\hookrightarrow R$ is integral, so $\m=\p \cap A$. Hence $\m R \subseteq \p$. Also, $x^n=a\in \p$ in $R$, so $x\in \p$ in $R$. Thus $xR+\m R\subseteq \p$. Since $xR+\m R$ is a maximal ideal of $R$, we have $\p=\m R+xR$ and hence $R$ is local.  

Since $\dim R=\dim A$, so $\m R$ is primary to the maximal ideal of $R$, so $R/\m R$ is artinian. This along with $\rmv(R/m R)\le 1$ implies $R/\m R$ is a hypersurface.  
\fi
\end{proof}  


\begin{prop}\label{p48} 
Let $H=\langle a_1, a_2, \dots, a_\ell \rangle$ be a numerical semigroup generated by positive integers $a_1, a_2, \dots, a_\ell$. Choose integers $m, n\in \mathbb{N}$ such that $\gcd(m,n)=1$ and $m\in H$ but $m\ne a_i$ for all $1 \le i \le \ell$. Set $S:=nH + \langle m \rangle$. Let
\begin{center}
 $A:=k[|H|]=k[|t^h \mid h\in H|]$ \quad and \quad $R:=k[|S|]=k[|t^s \mid s\in S|]$ 
\end{center}
denote the numerical semigroup rings of $H$ and $S$ over a field $k$ respectively, where $k[|t|]$ denotes the formal power series ring over $k$. Then we have the isomorphism 
\[
R\cong A[X]/(X^n-t^m)
\]
of rings, where $A[X]$ denotes the polynomial ring over $A$. Hence $R$ is a free $A$-module of rank $n$. Consequently, if $A$ satisfies {\rm (SAC)}, then so does $R$. 
\end{prop}  

\begin{proof}
The ring homomorphism $k[|t|] \to k[|t|]; t \mapsto t^n$ induces a ring homomorphism $A\to R; t^h \mapsto t^{nh}$. This also induces the surjective ring homomorphism $\varphi: A[X] \to R$ by corresponding $X \to t^m$. Checking that $(X^n-t^m) \subseteq \Ker \varphi$, we obtain the surjection $\overline{\varphi}: A[X]/(X^n-t^m) \to R$. Consider the commutative diagram 
\[
\xymatrix{
0 \ar[r] & Z:=\Ker \overline{\varphi}\ar[r] \ar[d]^{\cdot x} &T:=A[X]/(X^n-t^m) \ar[d]^{\cdot x} \ar[r]^{\quad \quad \quad \overline{\varphi}} & R\ar[r] \ar[d]^{\cdot t^m}  & 0\\
0 \ar[r] & Z\ar[r]  & T \ar[r]^{\quad \quad \quad \overline{\varphi}} & R \ar[r] & 0 \\
}
\]
of $T$-modules, where $x$ is the image of $X$ in $T$ and $\cdot x$ denotes the multiplication by $x$. By Snake lemma, we obtain the exact sequence 
\[
0 \to Z/XZ \to A/t^m A \to R/t^m R \to 0
\]
of $A$-modules. Note that $\ell_{A}(A/t^m A) = \ell_{A}(k[|t|]/t^m k[|t|]) = m$ and $\ell_{A}(R/t^m R)= \ell_{A}(k[|t|]/t^m k[|t|]) =m$ by the fact that $A$, $R$, and $k[|t|]$ are finitely generated $A$-modules of rank one and \cite[Theorem 14.8]{Mat}. Hence we have $\ell_{A}(Z/XZ) =0$. By Nakayama's lemma, we obtain that $Z=0$. Therefore $\overline{\varphi}: A[X]/(X^n-t^m) \to R$ is an isomorphism. 

Now let $\m$ be the maximal ideal of $A$. Then $(\m,x)R$ is the unique maximal ideal of $R$ by Lemma \ref{l47}(1). Thus, $R\cong R_{(\m,x)}\cong \dfrac{A[X]_{(\m,X)}}{(X^n-t^m)A[X]_{(\m,X)}}$. Therefore, if $A$ satisfies (SAC), then $A[X]_{(\m,X)}$ also satisfies (SAC) by Corollary \ref{localiz}. Since $A[X]_{(\m,X)}$ is an integral domain, $R\cong \dfrac{A[X]_{(\m,X)}}{(X^n-t^m)A[X]_{(\m,X)}}$ satisfies (SAC) by Corollary \ref{mainseccor3}. 
\end{proof}

\begin{Example} Let $e\ge 2$ be an integer, and set $H=\langle e,e+1,\dots,2e-1 \rangle$. Then, $R=k[[H]]$ has minimal multiplicity, hence satisfies (SAC). Then, for any integer $m\ge 2e$ and $n$ such that $\text{gcd}(m,n)=1$, we have that $k[[nH+\langle m\rangle ]]$ satisfies (SAC) by Proposition \ref{p48}. 
\end{Example}

\begin{Example} Let $e\ge 3$ be an integer, and set $H=\langle e,e+1,\dots,2e-2 \rangle$. Then $R=k[[H]]$ has almost minimal multiplicity, i.e, $\rme(R)=\rmv(R)+1$. This equality is equivalent to saying that $\ell_R(\fkm^2/x\fkm)=1$. Hence $\m^3\subseteq x \m$ for a minimal reduction $x$ of $\fkm$, where $\fkm$ denotes the unique maximal ideal of $R$. Consequently $R/(x)$ has radical cube zero. Then $R/(x)$ satisfies (SAC) by  \cite[Theorem 4.1]{JS}. Hence $R$ satisfies (SAC) by Corollary \ref{a0.3}. Then for any integer $m\ge 2e$ (so that $m\in H$ but $m$ is not a minimal generator of $H$) and $n$ such that $\text{gcd}(m,n)=1$, we have that $k[[nH+\langle m\rangle ]]$ satisfies (SAC) by Proposition \ref{p48}. 
\end{Example}

\section{Appendix: Modified Diveris' finitistic extension degree}\label{section6}

In this appendix,  we introduce a new invariant which raised from the study of Diveris. In \cite[Definition 2.2]{D}, Diveris defines an invariant called ``finitistic extension degree". Diveris shows that for a Gorenstein ring $R$, the invariant is finite if and only if $R$ satisfies (SAC) (\cite[Corollary 3.2]{D}). Here, we introduce a slightly different invariant and show that it behaves well as finitistic extension degree. (We note that in what follows, a given ring $R$ need not necessarily  be local.)

\begin{dfn}\label{feddef} 
Let $R$ be a Noetherian ring and $M$ a finitely generated $R$-module. Then, 
\[
\extdeg(M):=\sup \{i \mid \Ext_R^i(M, M)\ne 0\}
\]
is called the {\it self-extension degree} of $M$ (\cite[Definition 2.2]{D}). Here, we define the {\it restricted finitistic extension degree} as follows.
\[
\rfed(R):=\sup\{\extdeg(M\oplus R)\mid \extdeg(M\oplus R)<\infty\}.
\] 
For an integer $n\ge 0$, we also define
\[
\rfed_n(R):=\sup\{\extdeg(M\oplus R)\mid \text{$M$ has constant rank $n$ such that } \extdeg(M\oplus R)<\infty\}.
\] 
\end{dfn}


\begin{prop}\label{fed} Let $R$ be a ring having finite Krull-dimension (for example, $R$ is local). Then the following are equivalent.\\
{\rm (1)} $R$ satisfies {\rm (SAC)}. \quad \quad {\rm (2)} $\rfed(R)<\infty$. \quad \quad {\rm (3)} $\rfed(R)\le \dim  R$. 
\end{prop} 

\begin{proof} 
(3) $\Rightarrow$ (2): This is clear. 

(2) $\Rightarrow$ (1): Assume that $\rfed(R)<\infty$ and set $r=\rfed(R)$. Let $M$ be a finitely generated $R$-module such that $\Ext^{>0}_R(M,R)=\Ext^{\gg 0}_R(M,M)=0$. 
We then have $\Ext^{\gg 0}_R(\syz_R^{s}M,\syz_R^{s}M)=\Ext^{\gg 0}_R(M,\syz_R^{s}M)=0$ and $\Ext^{\gg 0}_R(\syz_R^{s}M,R)=0$ for all $s>0$. Thus, $\extdeg(R\oplus \syz_R M \oplus \syz_R^{r+2}M)<\infty$. Now $\rfed(R)=r<\infty$ implies $\Ext^{r+1}_R(R\oplus \syz_R^1 M \oplus \syz_R^{r+2}M,R\oplus \syz_R^1 M \oplus \syz_R^{r+2}M)=0$. Hence,  
\[
\Ext^{r+1}_R(\syz_R^1 M,\syz_R^{r+2}M)=0=\Ext^1_R(\syz_R^{r+1}M,\syz_R^1 \syz_R^{r+1}M)=0.
\] 
It follows that $0 \to \syz_R^{r+1} M \to P_r \to \syz_R^{r} M \to 0$ splits, where $P_r$ is a finitely generated projective $R$-module. Hence, $\syz_R^{r+1} M$ is projective and hence $\pd_R M<\infty$. Now $\Ext^{>0}_R(M,R)=0$ implies $M$ is projective by Remark \ref{rem1}. 

(1) $\Rightarrow$ (3): Assume that $R$ satisfies (SAC). If $\extdeg(M\oplus R)<\infty$ then $\Ext^{\gg 0}_R(M,R\oplus M)=0$. By localizing at a prime ideal $\fkp$ of $R$ and by Proposition \ref{b2.1}, we get $\pd_{R_{\p}} M_{\p}<\infty$ for all $\p \in \spec(R)$. Consequently, 
\[
\pd_R M=\sup\{\pd_{R_{\p}} M_{\p}:\p \in \spec(R)\}\le \sup\{\depth R_{\p}: \p \in \spec(R)\}\le \dim R,
\] 
where the first inequality follows from the Auslander-Buchsbaum formula. This concludes that $\Ext^{>\dim R}_R(M\oplus R, M\oplus R)=0$ and thus $\rfed(R)\le \dim R<\infty$. 
\end{proof}  

We record a lemma, which follows by applying $\Hom_R(M,-)$ to an exact sequence $0\to \syz_R N \to F \to N \to 0$,  where $F$ is a finitely generated projective $R$-module, and using standard syzygy shifting argument. This lemma is needed for proving our next result, and will also be used in the proof of the main result of the next section.   

\begin{lem}\label{6} 
Let $R$ be a ring and $M,N$ finitely generated $R$-modules. Then, the following hold.
\begin{enumerate}[\rm(1)]
    \item If $\Ext^{\gg 0}_R(M,N\oplus R)=0$, then $\Ext^{\gg 0}_R(\syz^s_R M,  \syz^n_R N \oplus R)=0$ for all $s,n\ge 0$.
    
    \item   If $\Ext^{>0}_R(M,N\oplus R)=0$, then $\Ext^{>0}_R(\syz^n_R M,\syz^n_R N \oplus R)=0$ for all $n\ge 0$.
\end{enumerate}  
\end{lem}

\begin{prop}\label{fedn} Let $R$ be a ring having finite Krull-dimension (for example, $R$ is local). Then, the following are equivalent.  
\begin{enumerate}[\rm(1)]
\item $R$ satisfies {\rm (SACC)}. 

\item $\rfed_n(R)<\infty$ for every integer $n\ge 0$. 

\item $\rfed_n(R)<\infty$ for some  integer $n\ge 0$.  

\end{enumerate}  
\end{prop}  

\begin{proof} $(1) \Rightarrow (2)$: We assume that $R$ satisfies (SACC). 
Let $M$ be a finitely generated module with constant rank and $\extdeg (M\oplus R)<\infty$. Then $\Ext^{> j}_R(M,M\oplus R)=0$ for some $j>0$. Thus $\Ext^{>0}_R(\syz_R^j M,R)=0=\Ext^{>0}_R(\syz_R^jM,M)$ and hence $\Ext^{\gg 0}_R(\syz_R^j M,\syz_R^{j+1}M)=0$ follows from Lemma \ref{6}. Since $\syz_R^j M$ has constant rank and $R$ satisfies (SACC), we get $\Ext^{1}_R(\syz_R^jM,\syz_R^{j+1}M)=0$. Thus $\syz_R^jM$ is projective and hence $M$ has finite projective dimension. Thus $\extdeg(M\oplus R)\le \pd_R M\le \dim R$. It follows that $\rfed_n(R)\le \dim R<\infty$.

$(2)\Rightarrow (3)$ is obvious.   

$(3) \Rightarrow (1)$: Assume that there exists an integer $n\ge 0$ such that $\rfed_n(R)<\infty$ and we prove that $R$ satisfies (SACC).  Let $r=\rfed_n(R)$ and $M$ be a finitely generated $R$-module of  constant rank such that $\Ext^{>0}_R(M,R)=0=\Ext^{\gg 0}_R(M,M)$. We prove that $M$ has finite projective dimension. Suppose the contrary. 
Then $\syz_R^jM\ne 0$ for all $j>0$  and 	$\Ext_R^{>0}(\om^j_R M, R)=0= \Ext_R^{\gg 0}(\om^j_R M, \om^j_R M)$ by Lemma \ref{6}. For $j\geq 0$, set $e_j=\text{rank } \om^j_R M$. Note that $\syz_R^j M$ is torsion-free and non-zero for each $j>0$; hence, $e_j\ge 1$.   
	Let $X:=(\om^1_R M)^{\oplus n}\oplus \om^{r+3}_R M$. Then $X$ has rank $e:=ne_1+e_{r+3}\geq n+1$ and $X$ has a free submodule of rank $e$ (\cite[Proposition 1.4.3]{BH}). Consider a short exact sequence of $R$-modules 
	\begin{align}\label{eq52}
	0\longrightarrow R^{\oplus (e-n)}\longrightarrow X\longrightarrow L\longrightarrow 0,
	\end{align}
	 where $L$ has rank $n$. 
 First applying $\Hom_R(X,-)$ to \eqref{eq52}, we get $\Ext^{\gg 0}_R(X,L)=0.$ Again applying $\Hom_R(-,L)$ to \eqref{eq52}, we get $\Ext^{\gg 0}_R(L,L)=0.$ Since $\Ext^{\gg 0}_R(X,R)=0$ implies $\Ext^{\gg 0}_R(L,R)=0$ and by the hypothesis that $r<\infty$, we get $\Ext^{\geq r+1}_R(L,L)=0.$ 
Applying  $\Hom_R(-,L)$ to \eqref{eq52}, we have $\Ext^{\geq r+1}_R(X,L)=0.$ Finally applying $\Hom_R(X,-)$ to \eqref{eq52}, we get $\Ext^{\geq r+2}_R(X,X)=0.$  Therefore, recalling $X=(\om^1_R M)^{\oplus n}\oplus \om^{r+3}_R M$, we obtain that  
$$0=\Ext^{r+2}_R(\om^1_R M,\om^{r+3}_R M)\cong \Ext^{1}_R(\om^{r+2}_R M,\om^{r+3}_R M).$$
This implies that $\om^{r+2}_R M$ is a projective $R$-module. This is a contradiction. 

Hence, $M$ has finite projective dimension. Then,  $\Ext^{>0}_R(M,R)=0$ implies $M$ is a projective $R$-module by Remark \ref{rem1}. Therefore, we have $\Ext^{> 0}_R(M,M)=0$. 
\end{proof}  

From Corollary \ref{b2.4} and Propositions \ref{fed} and \ref{fedn}, we get the following.

\begin{thm}\label{fedth} Let $R$ be a local ring of positive depth. Then the following are equivalent.
\begin{enumerate}[\rm(1)] 
\item $R$ satisfies {\rm (SAC)}.
\item $\rfed(R)<\infty$.
\item There is an integer $n\ge 0$ such that $\rfed_n(R)<\infty$.
\end{enumerate}

\end{thm}

\begin{ac}
The authors are grateful to Olgur Celikbas. He introduced the paper \cite{CT} to the authors and gave helpful comments.  
\end{ac}

\end{document}